\documentclass[11pt,leqno]{amsart}

\usepackage{amsmath}
\usepackage{amssymb}
\usepackage{a4wide}
\usepackage{graphics}
\usepackage{epsfig}
\usepackage{enumerate}
\newtheorem{theorem}{Theorem}[section]
\newtheorem{lemma}[theorem]{Lemma}
\newtheorem{proposition}[theorem]{Proposition}
\newtheorem{definition}[theorem]{Definition}
\newtheorem{remark}[theorem]{Remark} 
\newtheorem{corollary}[theorem]{Corollary}

\newcommand{\Subsection}[1]{\subsection{ #1} ${}^{}$}

\newcounter{hypo}

\def\C{{\mathbb C}}

\def\N{{\mathbb N}} 
\def\R{{\mathbb R}}

\def\CO{\mathcal {O}}

\def\one{{\mathchoice {\rm 1\mskip-4mu l} {\rm 1\mskip-4mu l} {\rm 1\mskip-4.5mu l} {\rm 1\mskip-5mu l}}}
\def\re{\mathop{\rm Re}\nolimits}
 \def\im{\mathop{\rm Im}\nolimits}

\def\Op{\mathop{\rm Op}\nolimits}

\def\supp{\mathop {\rm supp}\nolimits}

\def\ad{\mathop{\rm ad}\nolimits}
\def\<{\langle}
\def\>{\rangle}

\def\slim{\mathop{\text{s--lim}}}

\newcommand{\fract}[2]{\genfrac{}{}{0pt}{}{\scriptstyle #1}{\scriptstyle #2}}

\makeatletter
 \@addtoreset{equation}{section}
 \makeatother

\author[J.-F. Bony]{Jean-Fran\c{c}ois Bony}

\author[D. H\"{a}fner]{Dietrich H\"{a}fner}

\email{bony@math.u-bordeaux1.fr}
\email{hafner@math.u-bordeaux1.fr} 
\address{\newline Institut de Math\'ematiques de Bordeaux   \newline UMR 5251 du CNRS   \newline Universit\'e de Bordeaux I  \newline 351 cours de la Lib\'eration   \newline 33 405 Talence cedex    \newline France}

\title[The semilinear wave equation on asymptotically Euclidean manifolds]{The semilinear wave equation on asymptotically Euclidean manifolds}



\begin{document}

\begin{abstract}
We consider the quadratically semilinear wave equation on $( \R^d , \mathfrak{g})$, $d \geq 3$. The metric $\mathfrak{g}$ is non-trapping and approaches the Euclidean metric like $\< x \>^{- \rho}$. Using Mourre estimates and the Kato theory of smoothness, we obtain, for $\rho >0$, a Keel--Smith--Sogge type inequality for the linear equation. Thanks to this estimate, we prove long time existence for the nonlinear problem with small initial data for $\rho \geq 1$. Long time existence means that, for all $n>0$, the life time of the solution is a least $\delta^{-n}$, where $\delta$ is the size of the initial data in some appropriate Sobolev space. Moreover, for $d \geq 4$ and $\rho > 1$, we obtain global existence for small data.
\end{abstract}

\maketitle

\setcounter{tocdepth}{1}
\tableofcontents

\section{Introduction} \label{sec1}

This paper is devoted to the study of the quadratically semilinear wave equation on asymptotically Euclidean non-trapping Riemannian manifolds. We show global existence in dimension $d\ge 4$ and long time existence in dimension $d=3$ for small data solutions. In Minkowski space, the semilinear wave equation has been thoroughly studied. Global existence is known in dimension $d\ge 4$ for small initial data (see Klainerman and Ponce \cite{KlPo83_01} and references therein). Almost global existence in dimension $d=3$ for small data was shown by John and Klainerman in \cite{JoKl84_01}. Almost global means that the life time of a solution is at least $e^{1/ \delta}$, where $\delta$ is the size of the initial data in some appropriate Sobolev space. Note that, in dimension $d=3$, Sideris~\cite{Si83_01} has proved that global existence does not hold in general (see also John~\cite{Jo81_01}).

In \cite{KeSmSo02_01}, Keel, Smith and Sogge give a new proof of the almost global existence result in dimension 3 using estimates of the form
\begin{equation} \label{KSSE}
( \ln (2+ T ) )^{-1/2} \big\Vert \< x \>^{-1/2} u' \big\Vert_{L^2 ( [0,T] \times \R^3 )} \lesssim \Vert u' (0, \cdot ) \Vert_{L^2 ( \R^3 )} + \int_0^T \Vert G(s, \cdot ) \Vert_{L^2 ( \R^3 )} d s ,
\end{equation}
and a certain Sobolev type estimate due to Klainerman (see \cite{Kl85_01}). Here $u$ solves the wave equation $\Box u = G$ in $[0 , + \infty [ \times \R^3$ and $u' = ( \partial_{t} u , \partial_{x} u )$. They also treat the non-trapping obstacle case. In \cite{KeSmSo04_01}, similar results are obtained for the corresponding quasilinear equation. The obstacle case in which the trapped trajectories are of hyperbolic type is treated by Metcalfe and Sogge \cite{MeSo05_01}.

Alinhac shows an estimate similar to \eqref{KSSE} on a curved background. In his papers \cite{Al05_01} and \cite{Al06_01}, the metric is depending on and decaying in time. The results of Metcalfe and Tataru \cite{MeTa08_01} imply estimates analogous to \eqref{KSSE} for a space-time variable coefficients wave equation outside a star shaped obstacle (see also \cite{MeSo06_01}). Outside the obstacle, their wave operator is a small perturbation of the wave operator in Minkowski space.

The common point of the papers cited so far is that they all use vector field methods. We use in this paper a somewhat different approach. We will show how estimates of type \eqref{KSSE} follow from a Mourre estimate \cite{Mo81_01}. This method will permit us to consider non-trapping Riemannian metrics which are asymptotically Euclidean without requiring that they are everywhere a small perturbation of the Euclidean metric. We will suppose for simplicity that the metric is $C^{\infty}$, but a $C^k$ approach should in principle be possible. Spectral methods for proving dispersive estimates were previously used by Burq. In \cite{Bu03_01}, he obtains global Strichartz estimates for compactly supported non-trapping perturbations of the Euclidean case. In more complicated geometries, conjugate operators are probably not vector fields and it is perhaps worth trying to mix the classical vector field approach with the Mourre theory.

Let us now state our precise results. We consider the asymptotically Euclidean manifold $( \R^d , \mathfrak{g})$ with $d \geq 3$ and
\begin{equation*}
\mathfrak{g} = \sum_{i,j=1}^{d} g_{i,j} (x) \, d x^i \, d x^j .
\end{equation*}
We suppose $g_{i,j} (x) \in C^{\infty} ( \R^{d} )$ and, for some $\rho >0$,
\begin{equation}\tag{H1} \label{c1}
\forall \alpha \in \N^d \qquad \partial^{\alpha}_x ( g_{i,j} - \delta_{i,j} ) = \CO ( \< x \>^{- \vert \alpha \vert - \rho} ) .
\end{equation} 
We also assume that
\begin{equation}\tag{H2} \label{c15}
\mathfrak{g} \text{ is non-trapping.}
\end{equation} 
Let $g (x) = ( \det ( \mathfrak{g} ) )^{1/4}$. The Laplace--Beltrami operator associated to $\mathfrak{g}$ is given by
\begin{equation*}
\Delta_{\mathfrak{g}} = \sum_{i,j} \frac{1}{g^2} \partial_i g^{i,j} g^2 \partial_j ,
\end{equation*}
where $g^{i,j} (x)$ denotes the inverse metric. Let us consider the following unitary transform
\begin{equation*}
\begin{aligned}
{\mathcal V}: \\
{}
\end{aligned}
\left\{ \begin{aligned}
&L^2 ( \R^d , g^{2} \, d x ) &\longrightarrow &&L^2( \R^d , d x )  \\
&v &\longmapsto  &&g v .
\end{aligned} \right.
\end{equation*}
The transformation ${\mathcal V}$ sends $- \Delta_{\mathfrak{g}}$ to
\begin{equation*}
P = - {\mathcal V} \Delta_{\mathfrak{g}} {\mathcal V}^* = - \sum_{i,j} \frac{1}{g} \partial_i g^{i,j} g^2 \partial_j \frac{1}{g} ,
\end{equation*}
which is the operator we are interested in. Let $\widetilde{\partial}_j : = \partial_j g^{-1}$ and $\Omega = \Omega^{k , \ell} : = x_{k} \partial_{\ell} - x_{\ell} \partial_{k}$ be the rotational vector fields. We consider the following semilinear wave equation
\begin{equation} \label{SLW}
\left\{\begin{aligned}
&\Box_{\mathfrak{g}} u = Q (u') ,  \\
&( u_{\vert_{t=0}} , \partial_t u _{\vert_{t=0}} ) = ( u_0 , u_1 ) .
\end{aligned} \right.
\end{equation}
Here $\Box_{\mathfrak{g}} = \partial_{t}^{2} + P$ and $Q (u')$ is a quadratic form in $u ' = ( \partial_t u , \widetilde{\partial}_{x} u )$. For $x \in \R$, $\lfloor x \rfloor$ (resp. $\lceil x \rceil$) denotes the largest (resp. smallest) integer such that $\lfloor x \rfloor \leq x \leq \lceil x \rceil$. Our main result is the following theorem.

\begin{theorem}\sl  \label{TSLW}
Assume hypotheses \eqref{c1} and \eqref{c15}. Suppose $u_0 , u_1 \in C_0^{\infty} ( \R^d )$ and that, for $M = 2 \left( \left\lceil \frac{d-1}{2} \right\rceil + 1 \right)$, we have
\begin{equation}  \label{SW1}
\sum_{\vert \alpha \vert + j \leq M+1} \big\Vert \partial^j_x \Omega^{\alpha} u_0 \big\Vert + \sum_{\vert \alpha \vert + j \leq M} \big\Vert \partial^j_x \Omega^{\alpha} u_1 \big\Vert \leq \delta .
\end{equation}

$i)$ Assume $d \geq 3$ and $\rho \geq 1$. For all $n > 0$, there exists a constant $\delta_{n} > 0$ such that, for $\delta \leq \delta_{n}$, the problem \eqref{SLW} has a unique solution $u \in C^{\infty}( [ 0, T ] \times \R^3 )$ with 
\begin{equation*}
T = \delta^{-n} .
\end{equation*}

$ii)$ Assume $d\geq 4$ and $\rho > 1$. For $\delta$ small enough, the problem \eqref{SLW} has a unique global solution $u \in C^{\infty} ( [ 0 , + \infty [ \times \R^d )$.
\end{theorem}

\begin{remark}\sl
One may consider more general nonlinearities. For example, the previous result holds for quadratic nonlinearities of the form $Q (x) ( \< x \>^{- \mu} u , u')$ with $\mu >1$ and $\Vert \partial_{x}^{\alpha} Q (x) \Vert= \CO ( \< x \>^{- \vert \alpha \vert} )$. In particular, one can replace $Q (u')$ by $Q ( \partial_{t} u , \partial_{x} u)$ or work with the wave equation before the transformation by ${\mathcal V}$. To prove this remark, it is enough to combine the proof of Theorem~\ref{TSLW} with Lemma~\ref{c16}.
\end{remark}

The main ingredient of the proof are estimates of type \eqref{KSSE}. Let us therefore consider the corresponding linear equation. Let $u$ be solution of
\begin{equation} \label{LW}
\left\{ \begin{aligned}
&(\partial_t^2+P)u = G(s) ,  \\
&( u_{\vert_{t=0}} , \partial_t u _{\vert_{t=0}} ) = ( u_0 , u_1 ) .
\end{aligned} \right.
\end{equation}
With the notation
\begin{equation*}
F^{\varepsilon}_{\mu} (T) = \left\{ \begin{aligned}
&T^{1-2 \mu +2 \varepsilon} &&\mu \leq 1/2 ,   \\
&1 &&\mu > 1/2 ,
\end{aligned} \right.
\end{equation*}
we have the following estimate.

\begin{theorem}\sl \label{TSLW2}
Assume that \eqref{c1} and \eqref{c15} hold with $\rho > 0$ and let $0 < \mu \leq 1$. For all $\varepsilon >0$, the solution of \eqref{LW} satisfies
\begin{equation} \label{EW2}
\big\Vert \< x \>^{-\mu} u' \big\Vert_{L^2 ( [0, T] \times \R^{d} )} \lesssim \< F_{\mu}^{\varepsilon} (T) \>^{1/2} \bigg( \Vert u' ( 0 , \cdot ) \Vert_{L^2 ( \R^{d} )} + \int_0^T \Vert G(s, \cdot ) \Vert_{L^2 ( \R^{d} )} d s \bigg).
\end{equation}
\end{theorem}

To prove the nonlinear theorem, it will be useful to have higher order estimates. To this purpose, let us put $\widetilde{\Omega}^{k, \ell} = x_k \widetilde{\partial}_\ell - x_\ell \widetilde{\partial}_k$, $Z=\{\partial_t , \widetilde{\partial}_x, \widetilde{\Omega}\}$, $Y= \{ \widetilde{\partial}_x, \widetilde{\Omega}\}$, $X = \{ \widetilde{\partial}_x \}$, where $\{\widetilde{\Omega}\}$ (resp. $\{ \widetilde{\partial}_x \}$) are the collections of rotational vector fields (resp. partial derivatives with respect to space variables). Then, we have

\begin{theorem}\sl \label{TW1}
Assume that \eqref{c1} and \eqref{c15} hold with $\rho > 1$ and let $N > 0$ and $1/2 \leq \mu \leq 1$. For all $\varepsilon > 0$, the solution of \eqref{LW} satisfies
\begin{align}
\sup_{0 \leq t \leq T} \sum_{1 \leq k + j \leq N+1} \big\Vert \partial_{t}^{k} P^{j/2} u (t & , \cdot ) \big\Vert_{L^2 ( \R^{d} )} + \sum_{\vert \alpha \vert \leq N} \< F_{\mu}^{\varepsilon}(T) \>^{-1} \big\Vert \<x\>^{-\mu} Z^{\alpha} u ' \big\Vert_{L^2( [0, T] \times \R^{d} )} \nonumber \\
& \lesssim \sum_{\vert \alpha \vert \leq N} \bigg( \big\Vert (Z^{\alpha}u) ' (0, \cdot ) \big\Vert_{L^2 ( \R^{d} )} + \int_0^T \big\Vert Z^{\alpha}G(s, \cdot ) \big\Vert_{L^2 ( \R^{d} )} d s \bigg).  \label{W1}
\end{align}
Moreover, for $\rho = 1$, the same inequality holds with $\< F_{\mu}^{\varepsilon}(T) \>^{-1}$ replaced by $\< T \>^{- \varepsilon}$.
\end{theorem}

\begin{remark}\sl \label{RW1}
$i)$ Note that, in Theorem \ref{TSLW} and Theorem \ref{TW1}, $\rho \geq 1$ is required whereas Theorem \ref{TSLW2} is valid under a general long range condition $\rho > 0$.

$ii)$ Theorem \ref{TSLW2} and Theorem \ref{TW1} remain valid if we replace $u '$ by $( \partial_t u , P^{1/2} u )$.
\end{remark}

The paper is organized in the following way. In Section \ref{c20}, we show scattering estimates in a general setting. Section \ref{secM} is devoted to the Mourre estimate for the wave equation on our asymptotically Euclidean manifold. Using these results, we prove the estimates for the linear wave equation (Theorem \ref{TSLW2} and Theorem \ref{TW1}) in Section \ref{EW}. From these estimates, we deduce the nonlinear result in Section \ref{sec6}. Appendix \ref{a29} collects some regularity properties of operators and Appendix \ref{b56} contains low frequency resolvent estimates.

\section{The general setting} \label{c20}

In this section, we obtain some abstract estimates which will be used to prove Theorem~\ref{TSLW2} and Theorem \ref{TW1}. These estimates are not specific to the wave equation and could help to show analogous estimates for other equations. The key ingredients are the limiting absorption principle and the Kato theory of smoothness.

We begin this section with the notion of regularity with respect to an operator. A full presentation of this theory can be found in the book of Amrein, A. Boutet de Monvel and Georgescu \cite{AmBoGe96_01}. In Appendix \ref{a29}, we recall the properties which will be used in this paper.

\begin{definition}\sl
Let $(A, D(A))$ and $(H,D(H))$ be self-adjoint operators on a separable Hilbert space ${\mathcal H}$. The operator $H$ is of class $C^{k} (A)$ for $k >0$, if there is $z \in \C \setminus \sigma ( H)$ such that
\begin{equation*}
\R \ni t \longrightarrow e^{i t A} (H-z)^{-1} e^{- i t A},
\end{equation*}
is $C^{k}$ for the strong topology of ${\mathcal L} ( {\mathcal H} )$.
\end{definition}

Let $H \in C^{1} (A)$ and $I\subset \sigma (H)$ be an open interval. We assume that $A$ and $H$ satisfy a Mourre estimate on $I$:
\begin{equation} \label{a3}
\one_{I} (H) i [ H , A ] \one_{I} (H) \geq \delta \one_{I} (H) ,
\end{equation}
for some $\delta >0$. As usual, we define the multi-commutators $\ad_{A}^{j} B$ inductively by $\ad_{A}^{0} B =B$ and $\ad_{A}^{j+1} B = [ A , \ad_{A}^{j} B ]$.

\begin{theorem}[Limiting absorption principle]\sl \label{a4}
Let $H \in C^{2} (A)$ be such that $\ad_{A}^{j} H$, $j=1,2$, are bounded on ${\mathcal H}$. Assume furthermore \eqref{a3}. Then, for all closed intervals $J \subset I$ and $\mu >1/2$, there exists $C_{J,\mu} >0$ such that
\begin{equation} \label{a5}
\sup_{\fract{\re z \in J}{\im z \neq 0}} \big\Vert \< A \>^{- \mu} ( H -z)^{-1} \< A \>^{-\mu} \big\Vert \leq C_{J , \mu} .
\end{equation}
\end{theorem}

If $A$ and $H$ depend on a parameter, the constant in the limiting absorption principle can be specified according to this parameter. In fact, mimicking the proof of \cite{PeSiSi81_01}, we obtain the following estimate.

\begin{remark}\sl \label{a6}
Assume that \eqref{a3} holds uniformly and that $[H ,A]$ is uniformly bounded. Then, for all closed intervals $J \subset I$ and $\mu >1/2$,
\begin{equation*}
\sup_{\fract{\re z \in J}{\im z \neq 0}}  \big\Vert \< A \>^{- \mu} ( H -z)^{-1} \< A \>^{-\mu} \big\Vert \leq \widetilde{C}_{J , \mu} \big\< \big\Vert \ad_{A}^{2} H \big\Vert \big\>^{\widetilde{C}_{J , \mu}} ,
\end{equation*}
for some $\widetilde{C}_{J, \mu} >0$.
\end{remark}

We now state a result of Kato \cite{Ka66_01} which says that, under the conclusions of Theorem \ref{a4}, $\< A \>^{- \mu} \one_{J} (H)$ is $H$--smooth. For the proof and more details, we refer to Theorem~XIII.25 and Theorem~XIII.30 of \cite{ReSi78_01}.

\begin{theorem}[$H$--smoothness]\sl \label{a1}
Let $A$ and $H$ be two self-adjoint operators satisfying \eqref{a5}. Then, for all closed intervals $J \subset I$ and $\mu > 1/2$,
\begin{equation*}
\int_{\R} \big\Vert \< A \>^{- \mu} e^{-i t H} \one_{J} (H) u \big\Vert^{2} d t \leq 8 C_{J , \mu} \Vert u \Vert^{2} ,
\end{equation*}
for all $u \in {\mathcal H}$.
\end{theorem}

In the previous theorem, $C_{J , \mu}$ is the constant appearing in \eqref{a5}. By interpolation, we get

\begin{corollary}\sl \label{a2}
Assume \eqref{a5}. Then, for all closed intervals $J \subset I$ and $0 < \mu \leq 1/2$,
\begin{equation*}
\int_{0}^{T} \big\Vert \< A \>^{- \mu} e^{-i t H} \one_{J} (H) u \big\Vert^{2} d t \leq M_{J , \mu , \varepsilon} T^{1 - 2 \mu + \varepsilon} \Vert u \Vert^{2} ,
\end{equation*}
for all $0 < \varepsilon < 2 \mu$. Here,
\begin{equation*}
M_{J , \mu , \varepsilon} = \big( 8 C_{J , \mu / ( 2 \mu - \varepsilon )} \big)^{2 \mu - \varepsilon} .
\end{equation*}
\end{corollary}

\begin{proof}
Since $e^{-i t H}$ is unitary,
\begin{equation*}
\int_{0}^{T} \big\Vert e^{-i t H} \one_{J} (H) u \big\Vert^{2} d t \leq T \Vert u \Vert^{2} .
\end{equation*}
Combining Theorem \ref{a1}, the previous estimate and an interpolation argument, we get
\begin{equation*}
\int_{0}^{T} \big\Vert \< A \>^{- (1- \theta ) \nu} e^{-i t H} \one_{J} (H) u \big\Vert^{2} d t \leq ( 8 C_{J , \nu} )^{1 - \theta} T^{\theta} \Vert u \Vert^{2} .
\end{equation*}
for all $0 \leq \theta \leq 1$ and $\nu >1/2$. Taking $\theta = 1 - 2 \mu + \varepsilon \in [0 ,1]$ (since $\varepsilon < 2 \mu$) and $\nu = \mu / (1 - \theta ) = \mu / ( 2 \mu - \varepsilon ) > 1/2$, the corollary follows.
\end{proof}

We now study the non-homogeneous equation using the Fourier transform. Let $G (t) \in L^{1}_{\text{loc}} ( \R_{t}; {\mathcal H} )$ be such that $\supp G \subset [0 , + \infty [$. We consider the solution $u$ of 
\begin{equation} \label{a34}
\left\{ \begin{aligned}
&(i \partial_{t} - H )u (t) = \varphi (H) G (t) ,  \\
&u_{\vert_{t=0}} = 0 ,
\end{aligned} \right.
\end{equation}
with $\varphi \in L^{\infty} ( \R )$ and $\supp \varphi \subset J$. This means that
\begin{equation} \label{a31}
u (t) = - i \int_{0}^{t} e^{- i (t-s) H} \varphi (H) G (s) \, d s ,
\end{equation}
and then $u \in C^{0} (\R_{t} ; {\mathcal H} ) \cap {\mathcal S} ' (\R_{t} ; {\mathcal H})$.

\begin{lemma}\sl \label{a32}
Let $A$ and $H$ be two self-adjoint operators satisfying \eqref{a5}. Then, for all $\mu >1/2$ and $\varphi \in C^{1} ( \R )$ satisfying $\Vert \varphi \Vert_{\infty} \leq 1$, $\Vert \varphi ' \Vert_{\infty} \leq C_{1}$ and $\supp \varphi \subset J$, we have
\begin{equation*}
\big\Vert \< A \>^{- \mu} \varphi ( H) ( H-z)^{-1} \< A \>^{- \mu} \big\Vert \leq C_{J , \mu} + C_{1} ,
\end{equation*}
for all $z \in \C \setminus \R$.
\end{lemma}

\begin{proof}
Using Taylor's expansion formula, we have
\begin{equation*}
\varphi (x) = \varphi (y) + (x - y) \int_{0}^{1} \varphi ' ( t x + (1 - t) y ) \, d t ,
\end{equation*}
and then
\begin{equation*}
\varphi (H) (H -z)^{-1} = \varphi ( \re z ) (H - z)^{-1} + \int_{0}^{1} \varphi ' \big( t H + (1-t) \re z \big) \, d t \, (H - \re z) (H -z)^{-1} .
\end{equation*}
Using the spectral theorem, we obtain the following estimates:
\begin{align*}
\Big\Vert \int_{0}^{1} \varphi ' \big( t H + (1-t) \re z \big) \, d t \Big\Vert &\leq \int_{0}^{1} \big\Vert \varphi ' \big( t H + (1-t) \re z \big) \big\Vert \, d t \leq C_{1} ,  \\
\big\Vert (H - \re z ) (H -z)^{-1} \big\Vert &\leq \sup_{x \in \R} \big\vert x ( x - i \im z)^{-1} \big\vert \leq 1 .
\end{align*}

Therefore, for $\re z \in J$, we have
\begin{align*}
\big\Vert \< A \>^{- \mu} \varphi ( H) ( H-z)^{-1} \< A \>^{- \mu} \big\Vert &\leq \vert \varphi ( \re z) \vert \big\Vert \< A \>^{- \mu} ( H-z)^{-1} \< A \>^{- \mu} \big\Vert + C_{1} \Vert \< A \>^{- \mu} \Vert^{2}   \\
&\leq C_{J , \mu} + C_{1} .
\end{align*}
On the other hand, for $\re z \notin J$, $\varphi ( \re z ) =0$ and then
\begin{equation*}
\big\Vert \< A \>^{- \mu} \varphi ( H) ( H-z)^{-1} \< A \>^{- \mu} \big\Vert \leq C_{1} \Vert \< A \>^{- \mu} \Vert^{2} \leq C_{1} .
\end{equation*}
The two last estimates give the lemma.
\end{proof}

\begin{proposition}\sl \label{a33}
Let $A$ and $H$ be two self-adjoint operators satisfying \eqref{a5} and $\varphi \in C^{1} ( \R )$ as in Lemma \ref{a32}. Then, for all $\mu > 1/2$ and $G (t) \in L^{2} ( \R_{t}; D ( \< A \>^{\mu} ) )$ with $\supp G \subset [0 , + \infty [$, the solution $u$ of \eqref{a34} satisfies
\begin{equation*}
\int_{0}^{\infty} \big\Vert \< A \>^{- \mu} u (t) \big\Vert^{2} d t \leq ( C_{J , \mu} + C_{1} )^{2} \int_{0}^{\infty} \big\Vert \< A \>^{\mu} G (t) \big\Vert^{2} d t .
\end{equation*}
\end{proposition}

\begin{proof}
Let $u_{\varepsilon} = ( 1+ i \varepsilon H)^{-1} u$. From \eqref{a31}, $u_{\varepsilon} \in C^{1} (\R_{t} ; {\mathcal H} ) \cap C^{0} ( \R_{t} ; D (H)) \cap {\mathcal S} ' (\R_{t} ; D (H) )$ and $u_{\varepsilon}$ is the solution of the problem
\begin{equation} \label{a30}
\left\{ \begin{aligned}
&(i \partial_{t} - H )u_{\varepsilon} (t) = ( 1+ i \varepsilon H )^{-1} \varphi (H) G (t) ,  \\
&u_{\varepsilon}{}_{\vert_{t=0}} = 0 .
\end{aligned} \right.
\end{equation}
Since the support of the temperate distributions $u_{\varepsilon}$ and $G$ is in $[0, + \infty [$, their Fourier transforms are analytic in $\im z < 0$. Then, \eqref{a30} gives, for $\im z < 0$,
\begin{equation*}
(z - H) \widehat{u_{\varepsilon}} (z) = (1 + i \varepsilon H )^{-1} \varphi (H) \widehat{G} (z) .
\end{equation*}
Then
\begin{equation*}
\< A \>^{- \mu} \widehat{u_{\varepsilon}} (z) = \< A \>^{- \mu} ( 1+ i \varepsilon H)^{-1} \varphi ( H) ( z-H)^{-1} \< A \>^{- \mu} \< A \>^{\mu} \widehat{G} (z) .
\end{equation*}
Since $\Vert \varphi ( x) ( 1+ i \varepsilon x )^{-1} \Vert_{\infty} \leq \Vert \varphi \Vert_{\infty}$ and $\Vert \partial_{x} ( \varphi ( x) ( 1+ i \varepsilon x )^{-1} ) \Vert_{\infty} \leq \varepsilon \Vert \varphi \Vert_{\infty} + \Vert \varphi ' \Vert_{\infty}$, Lemma \ref{a32} implies
\begin{equation*}
\big\Vert \< A \>^{- \mu} \widehat{u_{\varepsilon}} (z) \big\Vert \leq ( C_{J , \mu} + C_{1} + \varepsilon ) \big\Vert \< A \>^{\mu} \widehat{G} (z) \big\Vert .
\end{equation*}
Thus, for all $\delta >0$, Plancherel's theorem gives
\begin{equation*}
\int_{0}^{+ \infty} e^{- \delta t} \big\Vert \< A \>^{- \mu} u_{\varepsilon} \big\Vert^{2} d t \leq ( C_{J , \mu} + C_{1} + \varepsilon )^{2} \int_{0}^{+ \infty} e^{- \delta t} \big\Vert \< A \>^{\mu} G \big\Vert^{2} d t .
\end{equation*}
Letting $\delta$ and $\varepsilon$ go to $0$, we get the proposition.
\end{proof}

By interpolation, we also have

\begin{corollary}\sl \label{c9}
Assume the hypotheses of Proposition \ref{a33}. Then, for all $0 < \mu \leq 1/2$ and $G (t) \in L^{2} ( \R_{t}; D ( \< A \>^{\mu} ) )$ with $\supp G \subset [0 , + \infty [$,
\begin{equation*}
\int_{0}^{T} \big\Vert \< A \>^{- \mu} u (t) \big\Vert^{2} d t \leq N_{J , \mu , \varepsilon} T^{2 (1 - 2 \mu + \varepsilon )} \int_{0}^{T} \big\Vert \< A \>^{\mu} G (t) \big\Vert^{2} d t ,
\end{equation*}
for all $0 < \varepsilon < 2 \mu$. Here,
\begin{equation*}
N_{J , \mu , \varepsilon} = \big( C_{J , \mu / ( 2 \mu - \varepsilon )}  + C_{1} \big)^{4 \mu - 2 \varepsilon} .
\end{equation*}
\end{corollary}

\begin{proof}
Let $P_{T} : L^{2} ( [0, T] ; {\mathcal H} ) \longrightarrow L^{2} ( [0, T] ; {\mathcal H} )$ be the operator defined by
\begin{equation*}
( P_{T} G ) (t) = - i \int_{0}^{t} e^{- i (t-s) H} \varphi (H) G (s) \, d s .
\end{equation*}
Proposition \ref{a33} gives
\begin{equation*}
\Vert \< A \>^{- \nu} P_{T} G \Vert_{L^{2} ( [0, T] ; {\mathcal H} )} \leq ( C_{J , \nu} + C_{1} ) \Vert \< A \>^{\nu} G \Vert_{L^{2} ( [0, T] ; {\mathcal H} )} ,
\end{equation*}
for $\nu >1/2$. Moreover,
\begin{align*}
\Vert P_{T} G \Vert_{L^{2} ( [0, T] ; {\mathcal H} )} &\leq \sqrt{T} \sup_{t \in [0 , T]} \Big\Vert \int_{0}^{t} e^{- i (t-s) H} \varphi (H) G (s) \, d s \Big\Vert   \\
&\leq \sqrt{T} \sup_{t \in [0 , T]} \sqrt{t} \bigg( \int_{0}^{t} \Big\Vert e^{- i (t-s) H} \varphi (H) G (s) \Big\Vert^{2} \, d s \bigg)^{1/2}  \\
&\leq T \Vert G \Vert_{L^{2} ( [0, T] ; {\mathcal H} )} .
\end{align*}
With these two estimates in mind, one can prove the corollary by mimicking the proof of Corollary \ref{a2}.
\end{proof}

\section{The wave equation and the Mourre estimate}  \label{secM}

In this section we will show a Mourre estimate for the wave equation on our asymptotically Euclidean manifold:
\begin{equation}  \label{M3}
\left\{ \begin{aligned}
&( \partial_t^2 + P) u = 0 ,   \\
&u_{\vert_{t=0}} = u_0 , \ \partial_t u_{\vert_{t=0}} = u_1 .
\end{aligned} \right. 
\end{equation}
Recall that
\begin{equation}
P = - \sum_{i,j} \frac{1}{g} \partial_i g^{i,j} g^2 \partial_j \frac{1}{g} ,
\end{equation}
is self-adjoint on $L^2 ( \R^d , d x )$ with domain $D (P) = H^2 ( \R^{d} )$. We define $H^k_{\text{c}} ( \R^{d} )$ as the closure of $H^k ( \R^{d} )$ with respect to the norm
\begin{equation*}
\Vert u \Vert^2_{H_{\text{c}}^k} = \sum_{j=1}^k \big\Vert P^{j/2} u \big\Vert^2 .
\end{equation*}
Let ${\mathcal E} := H^1_{\text{c}} ( \R^{d} ) \oplus L^{2} ( \R^{d} )$ with
\begin{equation*}
\Vert ( u_0 , u_1 ) \Vert^{2}_{\mathcal E} = \< P u_0 , u_0 \> + \Vert u_1 \Vert^2 ,
\end{equation*}
be the energy space associated to \eqref{M3}. The energy of \eqref{M3} is clearly conserved:
\begin{equation*}
\big\Vert ( u(t) , \partial_t u(t) ) \big\Vert_{\mathcal E} = \Vert ( u_0 , u_1 ) \Vert_{\mathcal E} .
\end{equation*}
We will rewrite \eqref{M3} as a first order system
\begin{equation} \label{M4}
\left \{\begin{aligned}
&i \partial_t f = R f ,    \\
&f_{\vert_{t=0}} = ( u_0 , u_1 ) ,
\end{aligned} \right.
\end{equation}
with
\begin{equation*}
R = \left( \begin{array}{cc}
0 & i \\ - i P & 0
\end{array} \right) .
\end{equation*}
The operator $R$ is self-adjoint on ${\mathcal E}$ with domain $D (R) = H_{\text{c}}^2 ( \R^{d} ) \oplus H^{1} ( \R^{d} )$. Let ${\mathcal L} = L^{2} ( \R^{d} ) \oplus L^{2} ( \R^{d} )$. It is useful to introduce the following unitary transform:
\begin{equation*}
U : {\mathcal E} \longrightarrow {\mathcal L}, \qquad U = \frac{1}{\sqrt{2}}
\left( \begin{array}{cc}
P^{1/2} & i \\ P^{1/2} & -i
\end{array} \right) ,
\end{equation*}
which satisfies
\begin{equation*}
U^* = U^{-1} = \frac{1}{\sqrt{2}}
\left( \begin{array}{cc}
P^{-1/2} & P^{-1/2}  \\ -i & i
\end{array} \right)
\quad \text{ and } \quad L = U R U^* = \frac{1}{\sqrt{2}}
\left( \begin{array}{cc}
P^{1/2} & 0  \\
0 & -P^{1/2}
\end{array} \right) .
\end{equation*}
The operator $( L , D (L) = H^{1} ( \R^{d} ) \oplus H^{1} ( \R^{d} ) )$ is self-adjoint. In order to establish a Mourre estimate for $L$, it is sufficient to establish a Mourre estimate for $P^{1/2}$. We divide this section into the study of the low, the intermediate and the high frequency part.

\Subsection{Low frequency Mourre estimate}  \label{secM1}

For low frequencies, we will make a dyadic decomposition and use a conjugate operator specific to each part of the decomposition. In this section, we will obtain a Mourre estimate for each part. For $\lambda \geq 1$, we set
\begin{equation} \label{b29}
{\mathcal A}_{\lambda} = \varphi ( \lambda P ) A_{0} \varphi ( \lambda P) ,
\end{equation}
where
\begin{equation*}
A_{0}= \frac{1}{2} ( x D + D x ) , \quad D ( A_0 ) = \big\{ u \in L^{2} ( \R^{d} ) ; \ A_0 u \in L^{2} ( \R^{d} ) \big\} ,
\end{equation*}
is the generator of dilations and $\varphi \in C^{\infty}_{0} ( ] 0 , + \infty [ ; [0, + \infty [)$ satisfies $\varphi (x) > \delta >0$ on some open bounded interval $I \subset ] 0 , + \infty[$.

For the various estimates that we will establish in this section, the following formula for the square root of an operator will be useful. 
Making a change of contour and using the Cauchy formula, one can show that
\begin{equation} \label{c8}
\sigma^{-1/2} = \frac{1}{\pi} \int_{0}^{+ \infty} s^{-1/2} (s+ \sigma )^{-1} d s ,
\end{equation}
for $\sigma \neq 0$. Therefore, the functional calculus gives
\begin{equation} \label{b18}
\varphi ( \lambda P) P^{1/2} = \frac{1}{\pi} \int_{0}^{+ \infty} s^{-1/2} \varphi ( \lambda P) P (s+P)^{-1} d s  .
\end{equation}

It is well known that $P\in C^1(A_0)$. In particular, $\varphi ( \lambda P) : D ( A_0 ) \longrightarrow D ( A_0 )$ and ${\mathcal A}_{\lambda}$ is well defined on $D(A_0)$. Its closure, again denoted ${\mathcal A}_{\lambda}$, is self-adjoint (see \cite[Theorem~6.2.5, Lemma~7.2.15]{AmBoGe96_01}).

\begin{proposition}\sl  \label{PM1}
$i)$ We have $(\lambda P)^{1/2}\in C^2({\mathcal A}_{\lambda})$. The commutators $\ad^j_{{\mathcal A}_{\lambda}}(\lambda P)^{1/2}$, $j=1,2$, can be extended to bounded operators and we have, uniformly in $\lambda$,
\begin{align}
\big\Vert \big[ {\mathcal A}_{\lambda} , ( \lambda P)^{1/2} \big] \big\Vert & \lesssim 1 , \label{Mi1}  \\
\big\Vert \ad^2_{{\mathcal A}_{\lambda}} ( \lambda P)^{1/2} \big\Vert & \lesssim
\left\{\begin{aligned}
&1 && \rho > 1 , \\
&\lambda^{\varepsilon} && \rho \leq 1 ,
\end{aligned} \right.  \label{Mi2}
\end{align}
where $\varepsilon >0$ can be chosen arbitrary small. 

$ii)$ For $\lambda$ large enough, we have the following Mourre estimate:
\begin{equation}  \label{M7}
\one_{I} ( \lambda P) \big[ i ( \lambda P)^{1/2} , {\mathcal A}_{\lambda} \big] \one_{I} ( \lambda P) \geq \frac{\delta^{2} \sqrt{\inf I}}{2} \one_{I} ( \lambda P ) .
\end{equation}

$iii)$ For $0 \leq \mu \leq 1$ and $\psi \in C^{\infty}_{0} ( ] 0 , + \infty [ )$, we have
\begin{align}
\big\Vert \vert {\mathcal A}_{\lambda} \vert^{\mu} \< x \>^{- \mu} \big\Vert & \lesssim \lambda^{- \mu /2 + \varepsilon} ,  \\
\big\Vert \< {\mathcal A}_{\lambda} \>^{\mu} \psi ( \lambda P) \< x \>^{- \mu} \big\Vert & \lesssim \lambda^{- \mu /2 + \varepsilon} ,
\end{align}
for all $\varepsilon >0$.
\end{proposition}

The rest of this section will be devoted to the proof of the above proposition, which will be divided into several lemmas.

\begin{lemma}\sl  \label{LM2}
We have $(\lambda P)^{1/2} \in C^1 ( {\mathcal A}_{\lambda} )$.
\end{lemma}

The proof of Lemma \ref{LM2} is analogous to the proof of \cite[Lemme~3.3]{Ha01_01}. In this lemma, it is shown that $( ( \lambda P)^{1/2} , {\mathcal A}_{\lambda} )$ fulfill the original conditions of Mourre which imply the $C^1$ regularity. We now have to estimate the commutators. First note that 
\begin{equation*}
\big[ i ( \lambda P)^{1/2} , {\mathcal A}_{\lambda} \big] = i \lambda^{1/2} \varphi ( \lambda P) \big[ P^{1/2} , A_0 \big] \varphi ( \lambda P) .
\end{equation*}
Using formula \eqref{b18}, we find
\begin{equation} \label{b17}
\varphi ( \lambda P) \big[ P^{1/2} , A_{0} \big] \varphi ( \lambda P ) = \frac{1}{\pi} \int_{0}^{+ \infty} s^{-1/2} \varphi ( \lambda P) \big[ P (s+P)^{-1} , A_{0} \big] \varphi ( \lambda P) \, d s ,
\end{equation}
with
\begin{equation*}
\big[ P (s+P)^{-1} , A_{0} \big] = - s \big[ (s+P)^{-1} , A_{0} \big] = s (s+P)^{-1} \big[ P , A_{0} \big] (s+P)^{-1} .
\end{equation*}
A direct calculation gives
\begin{align}
\big[ P , A_{0} \big] =& - i \sum_{j,k} g^{-1} D_{j} \Big( 2 g^{2} g^{j,k} - \sum_{\ell} x_{\ell} \partial_{\ell} ( g^{2} g^{j,k} ) \Big) D_{k} g^{-1}     \nonumber  \\
&- i \sum_{\ell} g^{-2} x_{\ell} ( \partial_{\ell} g ) \sum_{j,k} D_{j} g^{2} g^{j,k}  D_{k} g^{-1} - i \sum_{j,k} g^{-1} D_{j} g^{2} g^{j,k}  D_{k} g^{-2} \sum_{\ell} x_{\ell} ( \partial_{\ell} g )  \nonumber  \\
=& - 2 i P + i \sum_{j,k} g^{-1} D_{j} \Big( \sum_{\ell} x_{\ell} \partial_{\ell} ( g^{2} g^{j,k} ) - 2 g^{j,k} g \sum_{\ell} x_{\ell} ( \partial_{\ell} g ) \Big) D_{k} g^{-1}    \nonumber \\
&+ \sum_{j, k, \ell} g^{-1} \partial_{j} \big( g^{-1} x_{\ell} ( \partial_{\ell} g ) \big) g^{2} g^{j,k} D_{k} g^{-1} - \sum_{j ,k, \ell} g^{-1} D_{j} g^{2} g^{j,k} \partial_{k} \big( g^{-1} x_{\ell} ( \partial_{\ell} g ) \big) g^{-1}   \nonumber \\
=& -2 i P - 2 i \sum_{j,k} g^{-1} D_{j} a_{j,k} D_{k} g^{-1} + 2 i \sum_{k} b_{k} D_{k} g^{-1} - 2 i \sum_{j} g^{-1} D_{j} b_{j} , \label{c18}
\end{align}
where $\partial^{\alpha}_{x} a_{j,k} = \CO ( \< x \>^{- \rho - \vert \alpha \vert} )$ and $\partial^{\alpha}_{x} b = \CO ( \< x \>^{- \rho - 1 - \vert \alpha \vert})$ by \eqref{c1}. In the following, a term $r_{j}$, $j\in \N$, will denote a smooth function such that
\begin{equation} \label{c17}
\partial^{\alpha}_{x} r_{j} (x) = \CO \big( \< x \>^{-\rho - j - \vert \alpha \vert} \big) .
\end{equation}
Moreover, to clarify the statement, we will not write the sums over $j$, $k$ and $j,k$ and replace the remainder terms in \eqref{c18} by $\widetilde{\partial}^{*} r_{0} \widetilde{\partial}$, $\widetilde{\partial}^{*} r_{1}$ and $r_{1} \widetilde{\partial}^{*}$. Then,
\begin{equation} \label{c7}
[P , A_{0} ] = - 2 i \big( P + \widetilde{\partial}^{*} r_{0} \widetilde{\partial} + \widetilde{\partial}^{*} r_{1} + r_{1} \widetilde{\partial} \big) .
\end{equation}
and \eqref{b17} becomes
\begin{align*}
\varphi ( \lambda P) \big[ P^{1/2} , A_{0} \big] \varphi ( \lambda P ) = - \frac{2 i}{\pi} \int_{0}^{+ \infty} & s^{1/2} \varphi ( \lambda P) (s+P)^{-1} \\
& \big( P + \widetilde{\partial}^{*} r_{0} \widetilde{\partial} + \widetilde{\partial}^{*} r_{1} + r_{1} \widetilde{\partial} \big) (s+P)^{-1} \varphi ( \lambda P) d s .
\end{align*}
Proceeding as in \eqref{b18}, one can show that
\begin{equation*}
\int_{0}^{+ \infty} s^{1/2} \varphi^{2} ( \lambda P ) P (s+P)^{-2} d s = \frac{\pi}{2} \varphi^{2} ( \lambda P ) P^{1/2} .
\end{equation*}
Then, we finally obtain
\begin{equation} \label{M7.1}
\big[ i ( \lambda P)^{1/2} , {\mathcal A}_{\lambda} \big] = ( \lambda P) ^{1/2} \varphi^{2} ( \lambda P ) + R ,
\end{equation}
with
\begin{equation}  \label{b27}
R = \frac{2}{\pi} \lambda^{1/2} \int_{0}^{+ \infty} s^{1/2} \varphi ( \lambda P) (s+P)^{-1} \big( \widetilde{\partial}^{*} r_{0} \widetilde{\partial} + \widetilde{\partial}^{*} r_{1} + r_{1} \widetilde{\partial} \big)   (s+P)^{-1} \varphi ( \lambda P) \, d s .
\end{equation}
The remainder term $R$ can be estimated in the following way.

\begin{lemma}\sl \label{LM3}
Assume $0 \leq \gamma \leq d/4$. Then, we have
\begin{equation*}
\Vert R u \Vert \lesssim \lambda^{- \gamma + \varepsilon} \big\Vert \< x \>^{- \min ( \rho , d/2) + 2 \gamma + \varepsilon} u \big\Vert ,
\end{equation*}
for all $\varepsilon >0$ .
\end{lemma}

\begin{proof}
First, we write
\begin{equation*}
\varphi ( \lambda P) (s+P)^{-1} \widetilde{\partial}^{*} = \lambda^{-1/2} (s+P)^{-1} \varphi ( \lambda P ) ( \lambda P )^{1/2} P^{-1/2} \widetilde{\partial}^{*} .
\end{equation*}
Using Lemma \ref{b53} and the functional calculus, this term can be estimated by
\begin{equation} \label{b26}
\big\Vert \varphi ( \lambda P) (s+P)^{-1} \widetilde{\partial}^{*} \big\Vert \lesssim \lambda^{-1/2} \big\Vert (s+P)^{-1} \one_{\supp \varphi} ( \lambda P ) \big\Vert \lesssim \lambda^{-1/2} (s + \lambda^{-1})^{-1} .
\end{equation}
Moreover, applying Lemma \ref{b61} (with $\beta=0$ and $\gamma = 1/2 \leq d/4$), we get
\begin{gather}
\big\Vert \varphi ( \lambda P) (s+P)^{-1} \widetilde{\partial}^{*} u \big\Vert \lesssim (s + \lambda^{-1})^{-1} \lambda^{-1 + \varepsilon} \Vert \< x \> u \Vert ,  \label{b54}   \\
\big\Vert \varphi ( \lambda P) (s+P)^{-1} u \big\Vert \lesssim (s + \lambda^{-1})^{-1} \lambda^{-1/2 + \varepsilon} \Vert \< x \> u \Vert . \label{b55}
\end{gather}

On the other hand, we write, for $k \in \N$,
\begin{equation*}
(s+P)^{-1} \varphi ( \lambda P) = \lambda ( \lambda s + \lambda P )^{-1} \psi ( \lambda P ) ( \lambda P + 1)^{-k} ,
\end{equation*}
with $\psi ( \sigma ) = \varphi ( \sigma ) ( \sigma +1 )^{k} \in C^{\infty}_{0} (]0 , + \infty [)$. Using the spectral theorem, we have
\begin{equation} \label{b22}
(s+P)^{-1} \varphi ( \lambda P) = \frac{\lambda}{\pi} \int (\lambda s + z )^{-1} \overline{\partial} \widetilde{\psi} (z) ( \lambda P -z)^{-1} (\lambda P +1)^{-k} L (d z) ,
\end{equation}
where $\widetilde{\psi} \in C^{\infty}_{0} ( \C )$ is an almost analytic extension of $\psi$. From the form of $\varphi$, one can always assume that $\supp \widetilde{\psi} \subset \{ z \in \C ; \ \re z > \widetilde{\varepsilon} >0 \}$. In particular,
\begin{equation} \label{b23}
\vert ( \lambda s + z )^{-1} \vert \lesssim ( \lambda s + 1)^{-1} ,
\end{equation}
uniformly for $z \in \supp \widetilde{\psi}$.

Using Proposition \ref{b45}, Lemma \ref{b20} and Remark \ref{b59}, we obtain
\begin{align}
\big\Vert \< x \>^{- \rho} ( \lambda P -z)^{-1} u \big\Vert &\lesssim \big\Vert \< x \>^{- \rho}( \lambda P -z)^{-1} \< x \>^{\min ( \rho , d/2) - \varepsilon} \< x \>^{- \min ( \rho , d/2) + \varepsilon} u \big\Vert   \nonumber  \\
&\lesssim \frac{1}{\vert \im z \vert^{C}} \big\Vert \< x \>^{- \min ( \rho , d/2) + \varepsilon} u \big\Vert , \label{b57}
\end{align}
and
\begin{align}
\big\Vert \< x \>^{- \rho} \widetilde{\partial} ( \lambda P -z)^{-1} u \big\Vert &\lesssim \big\Vert \< x \>^{- \rho} \widetilde{\partial} ( \lambda P -z)^{-1} \< x \>^{\min ( \rho , d/2) - \varepsilon} \< x \>^{- \min ( \rho , d/2) + \varepsilon} u \big\Vert   \nonumber  \\
&\lesssim \frac{\lambda^{-1/2}}{\vert \im z \vert^{C}} \big\Vert \< x \>^{- \min ( \rho , d/2) + \varepsilon} u \big\Vert , \label{b21}
\end{align}
for all $\varepsilon >0$.

Let $\gamma \leq d/4$ and fix $k > \gamma +2$. Applying $k$ times Proposition \ref{b45}, we get
\begin{equation} \label{b25}
\big\Vert \< x \>^{- \min ( \rho , d/2) + \varepsilon} (\lambda P +1)^{-k} u \big\Vert \lesssim \lambda^{- \gamma + \varepsilon} \big\Vert \< x \>^{- \min ( \rho , d/2) + 2 \gamma + 2 \varepsilon} u \big\Vert ,
\end{equation}
for all $\varepsilon >0$.

The formula \eqref{b22} and the estimates \eqref{b23}, \eqref{b57}, \eqref{b25} imply
\begin{align}
\big\Vert \< x \>^{- \rho} ( & s+P)^{-1} \varphi ( \lambda P) u \big\Vert   \nonumber  \\
&\lesssim \lambda ( \lambda s + 1 )^{-1} \int \vert \overline{\partial} \widetilde{\psi} (z) \vert \big\Vert \< x \>^{- \rho} ( \lambda P -z)^{-1} (\lambda P +1)^{-k} u \big\Vert \, L (d z)  \nonumber  \\
&\lesssim \lambda ( \lambda s + 1 )^{-1} \bigg( \int \vert \im z \vert^{-C} \vert \overline{\partial} \widetilde{\psi} (z) \vert \, L (d z) \bigg) \big\Vert \< x \>^{- \min ( \rho , d/2) + \varepsilon} (\lambda P +1)^{-k} u \big\Vert  \nonumber \\
&\lesssim \lambda^{1 - \gamma + \varepsilon} ( \lambda s + 1 )^{-1}  \big\Vert \< x \>^{- \min ( \rho , d/2) + 2 \gamma + 2 \varepsilon} u \big\Vert ,  \label{b58}
\end{align}
for all $\varepsilon >0$. The same way, using \eqref{b21} instead of \eqref{b57}, we obtain
\begin{equation}
\big\Vert \< x \>^{- \rho} \widetilde{\partial} (s+P)^{-1} \varphi ( \lambda P) u \big\Vert \lesssim \lambda^{1/2 - \gamma + \varepsilon} ( \lambda s + 1 )^{-1}  \big\Vert \< x \>^{- \min ( \rho , d/2) + 2 \gamma + 2 \varepsilon} u \big\Vert ,  \label{b28}
\end{equation}
for all $\varepsilon >0$.

Let $R_{1}$ be the term of \eqref{b27} with $\widetilde{\partial}^{*} r_{0} \widetilde{\partial}$. Using $r_{0} = \CO ( \< x \>^{- \rho})$, \eqref{b26} and \eqref{b28}, we get
\begin{align}
\Vert R_{1} u \Vert &\lesssim \lambda^{1/2 - \gamma + \varepsilon} \bigg( \int_{0}^{+ \infty} s^{1/2} (s + \lambda^{-1})^{-1} ( \lambda s + 1 )^{-1} d s \bigg) \big\Vert \< x \>^{- \min ( \rho , d/2) + 2 \gamma + 2 \varepsilon} u \big\Vert \nonumber \\
&\lesssim \lambda^{- \gamma + \varepsilon} \bigg( \int_{0}^{+ \infty} s^{1/2} (s + 1)^{-1} ( s + 1 )^{-1} d s \bigg) \big\Vert \< x \>^{- \min ( \rho , d/2) + 2 \gamma + 2 \varepsilon} u \big\Vert  \nonumber \\
&\lesssim \lambda^{- \gamma + \varepsilon} \big\Vert \< x \>^{- \min ( \rho , d/2) + 2 \gamma + 2 \varepsilon} u \big\Vert , 
\end{align}
for all $\varepsilon >0$. The same estimate can be proved for the term of \eqref{b27} with $\widetilde{\partial}^{*} r_{1}$ (resp. $r_{1} \widetilde{\partial}$) from $r_{1} = \CO ( \< x \>^{- \rho -1})$, \eqref{b54} and \eqref{b58} (resp. \eqref{b55} and \eqref{b28}).
\end{proof}

\begin{lemma}\sl \label{c11}
For all $\varepsilon >0$,
\begin{equation*}
\big[ \big[ ( \lambda P )^{1/2} , {\mathcal A}_{\lambda} \big] , {\mathcal A}_{\lambda} \big]= \CO ( \lambda^{\varepsilon} ) .
\end{equation*}
\end{lemma}

\begin{remark}\sl
If we assume $\rho >1$, Lemma \ref{c11} can be proved more simply. In fact, Lemma \ref{LM3} and Lemma \ref{LM4} give $\Vert R u \Vert \lesssim \Vert \< x \>^{-1} u \Vert$ and $\Vert {\mathcal A}_{\lambda} u \Vert \lesssim \Vert \< x \> u \Vert$. Using $R^{*} = R$, these estimates imply Lemma \ref{c11}.
\end{remark}

\begin{proof}
$\bullet$ We start with the commutator between ${\mathcal A}_{\lambda}$ and the first term on the right hand side of \eqref{M7.1}. Let $\psi ( \sigma ) = \sigma \varphi^{2} ( \sigma^{2})$ and $\widetilde{\psi}$ be an almost analytic extension of $\psi$. Then, we have
\begin{align*}
\big[ ( \lambda P )^{1/2} \varphi^{2} ( \lambda P ) , {\mathcal A}_{\lambda} \big] =& \big[ \psi \big( ( \lambda P )^{1/2} \big) , {\mathcal A}_{\lambda} \big]   \\
=& - \frac{1}{\pi} \int \overline{\partial} \widetilde{\psi} (z) \big( ( \lambda P )^{1/2} -z \big)^{-1} \big[ ( \lambda P )^{1/2} , {\mathcal A}_{\lambda} \big] \big( ( \lambda P )^{1/2} -z \big)^{-1} L ( d z ) .
\end{align*}
From \eqref{M7.1} and Lemma~\ref{LM3}, $[ ( \lambda P )^{1/2} , {\mathcal A}_{\lambda} ]$ is uniformly bounded. Therefore, the commutator $[ ( \lambda P )^{1/2} \varphi^{2} ( \lambda P ) , {\mathcal A}_{\lambda} ]$ is also uniformly bounded.

$\bullet$ We now study the commutator between ${\mathcal A}_{\lambda}$ and $R$ defined in \eqref{M7.1}. One can write
\begin{equation} \label{b65}
[ {\mathcal A}_{\lambda} , R ] = \big[ \varphi ( \lambda P ) , R \big] A_{0} \varphi ( \lambda P ) + \varphi ( \lambda P) [ A_{0} , R ] \varphi ( \lambda P ) + \varphi ( \lambda P ) A_{0} \big[ \varphi ( \lambda P ) , R \big],
\end{equation}
that we note $[ {\mathcal A}_{\lambda} , R ] = S_{1} + S_{2} + S_{3}$. Since $S_{3} = - S_{1}^{*}$, we only study the two first terms.

$\star$ With \eqref{b27} in mind, the operator $S_{1}$ can be written
\begin{align}
S_{1} = \frac{2}{\pi}\lambda^{1/2} \int_{0}^{+ \infty} s^{1/2} \varphi ( \lambda P) (s+P)^{-1} \big[ \varphi ( \lambda P ) , \widetilde{\partial}^{*} r_{0} \widetilde{\partial} + \widetilde{\partial}^{*} r_{1} + r_{1} \widetilde{\partial} \big] & \nonumber \\
(s+P)^{-1} \varphi ( \lambda P) A_{0} & \varphi ( \lambda P ) \, d s , \label{b66}
\end{align}
where
\begin{align}
\big[ \varphi( \lambda P ) , & \widetilde{\partial}^{*} r_{0} \widetilde{\partial} + \widetilde{\partial}^{*} r_{1} + r_{1} \widetilde{\partial} \big]   \nonumber  \\
&= - \frac{\lambda}{\pi} \int \overline{\partial} \widetilde{\varphi} (z) ( \lambda P -z)^{-1} \big[ P , \widetilde{\partial}^{*} r_{0} \widetilde{\partial} + \widetilde{\partial}^{*} r_{1} + r_{1} \widetilde{\partial} \big] ( \lambda P -z)^{-1} L ( d z ) , \label{b67}
\end{align}
and $\widetilde{\varphi}$ is an almost analytic extension of $\varphi$. A direct calculation gives
\begin{equation*}
\big[ P , \widetilde{\partial}^{*} r_{0} \widetilde{\partial} + \widetilde{\partial}^{*} r_{1} + r_{1} \widetilde{\partial} \big] = \widetilde{\partial}^{*} \widetilde{\partial}^{*} r_{1} \widetilde{\partial} + \widetilde{\partial}^{*} r_{2} \widetilde{\partial} + \widetilde{\partial}^{*} r_{3} + r_{3} \widetilde{\partial} ,
\end{equation*}
with the convention of \eqref{c17}. For the first term in this equality, we write
\begin{align*}
\widetilde{\partial}^{*} \widetilde{\partial}^{*} r_{1} \widetilde{\partial} =& \widetilde{\partial}^{*} ( \lambda P +1) ( \lambda P +1)^{-1} \widetilde{\partial}^{*} r_{1} \widetilde{\partial}   \\
=& ( \lambda P +1) \widetilde{\partial}^{*} ( \lambda P +1)^{-1} \widetilde{\partial}^{*} r_{1} \widetilde{\partial} - \lambda [ P , \widetilde{\partial}^{*} ] ( \lambda P +1)^{-1} \widetilde{\partial}^{*} r_{1} \widetilde{\partial} .
\end{align*}
As before, a direct calculation gives
\begin{equation*}
[ P , \widetilde{\partial}^{*} ] = \widetilde{\partial}^{*} r_{1} \widetilde{\partial} + \widetilde{\partial}^{*} r_{2} ,
\end{equation*}
with the usual decay on $r_{1}$ and $r_{2}$. Summing up, we get
\begin{align*}
\big[ P , \widetilde{\partial}^{*} r_{0} \widetilde{\partial} + \widetilde{\partial}^{*} & r_{1} + r_{1} \widetilde{\partial} \big] = ( \lambda P +1) \widetilde{\partial}^{*} ( \lambda P +1)^{-1} \widetilde{\partial}^{*} r_{1} \widetilde{\partial}  \\
&- \lambda \big( \widetilde{\partial}^{*} r_{1} \widetilde{\partial} + \widetilde{\partial}^{*} r_{2} \big) ( \lambda P +1)^{-1} \widetilde{\partial}^{*} r_{1} \widetilde{\partial} + \widetilde{\partial}^{*} r_{2} \widetilde{\partial} + \widetilde{\partial}^{*} r_{3} + r_{3} \widetilde{\partial}^{*} .
\end{align*}
Applying Lemma \ref{b20} (with $\beta = 1$ and $\gamma = 0$ satisfying $\gamma + \beta /2 \leq d/4$), Lemma \ref{b20} (with $\beta =0$ and $\gamma = 1/2$), Remark \ref{b59} and Lemma \ref{b53}, one can show that all the terms (say $\widetilde{r}$) of the last equation, with the exceptions of $\widetilde{\partial}^{*} r_{3}$ and $r_{3} \widetilde{\partial}$, satisfy
\begin{equation} \label{b68}
\big\Vert (\lambda P + 1)^ {-1} (\lambda P -z )^ {-1} \widetilde{r} (\lambda P -z )^ {-1} u \big\Vert \lesssim \frac{\lambda^{- 3/2 + \varepsilon}}{\vert \im  z \vert^{C}} \big\Vert \< x \>^{-1} u \big\Vert ,
\end{equation}
for all $\varepsilon >0$. Writing
\begin{equation*}
(\lambda P + 1)^ {-1} (\lambda P -z )^ {-1} \widetilde{\partial}^{*} r_{3} (\lambda P -z )^ {-1} = (\lambda P + 1)^ {-1} \< x \>^{-1} \< x \> (\lambda P -z )^ {-1} \widetilde{\partial}^{*} r_{3} (\lambda P -z )^ {-1} ,
\end{equation*}
and using Proposition \ref{b45} (with $\beta =0$ and $\gamma = 1/2$), Lemma \ref{b20} (with $\beta =1$ and $\gamma = 1/4$) and Proposition \ref{b45} (with $\beta =1$ and $\gamma = 1/4$), we get
\begin{equation} \label{b70}
\big\Vert (\lambda P + 1)^ {-1} (\lambda P -z )^ {-1} \widetilde{\partial}^{*} r_{3} (\lambda P -z )^ {-1} u \big\Vert \lesssim \frac{\lambda^{- 3/2 + \varepsilon}}{\vert \im  z \vert^{C}} \big\Vert \< x \>^{-1} u \big\Vert .
\end{equation}
Note that, in the case $d=3$, we have $\gamma + \beta /2 = d /4$. It is why we can not use the additional decay $ \< x \>^{- \rho}$ and loose $\lambda^{\varepsilon}$. In a similar manner, Proposition \ref{b45} (with $\beta = 0$ and $\gamma = 3/4$) and Lemma \ref{b20} (with $\beta =1$ and $\gamma = 1/4$) imply
\begin{equation} \label{b69}
\big\Vert (\lambda P + 1)^ {-1} (\lambda P -z )^ {-1} r_{3} \widetilde{\partial} (\lambda P -z )^ {-1} u \big\Vert \lesssim \frac{\lambda^{- 3/2 + \varepsilon}}{\vert \im  z \vert^{C}} \big\Vert \< x \>^{-1} u \big\Vert .
\end{equation}
Combining the estimates \eqref{b68}, \eqref{b70} and \eqref{b69} with the identity \eqref{b67}, we obtain
\begin{equation} \label{b71}
\big\Vert (\lambda P + 1)^ {-1} \big[ \varphi( \lambda P ) , \widetilde{\partial}^{*} r_{0} \widetilde{\partial} + \widetilde{\partial}^{*} r_{1} + r_{1} \widetilde{\partial} \big] u \big\Vert \lesssim \lambda^{- 1/2 + \varepsilon} \big\Vert \< x \>^{-1} u \big\Vert .
\end{equation}

From the form of $A_{0}$, we have
\begin{equation*}
A_{0} = - i g^{-1} \partial x g + a,
\end{equation*}
with $\partial_{x}^{\alpha} a (x) = \CO ( \< x \>^{- \vert \alpha \vert} )$. As in \eqref{b22}, we write
\begin{equation} \label{b72}
(s+P)^{-1} \varphi ( \lambda P) = \frac{\lambda}{\pi} \int (\lambda s + z )^{-1} \overline{\partial} \widetilde{\varphi} (z) ( \lambda P -z)^{-1} L (d z) .
\end{equation}
The above expression for $A_{0}$, together with Lemma \ref{b20} (with $\beta = 1$ and $\gamma =0$), Proposition \ref{b45} (with $\beta =0$ and $\gamma = 1/2$) and Remark \ref{b59}, gives
\begin{equation*}
\big\Vert \< x \>^{-1} ( \lambda P -z)^{-1} A_{0} \big\Vert \lesssim \frac{\lambda^{- 1/2 + \varepsilon}}{\vert \im z \vert^{C}} .
\end{equation*}
Then, \eqref{b72} (see also \eqref{b23}) implies
\begin{equation} \label{b73}
\big\Vert \< x \>^{-1} (s+P)^{-1} \varphi ( \lambda P ) A_{0} \big\Vert \lesssim \lambda^{1/2 + \varepsilon} (\lambda s +1)^{-1} ,
\end{equation}
for all $\varepsilon >0$.

Eventually, using the identity \eqref{b66}, the functional calculus and the estimates \eqref{b71} and \eqref{b73}, we obtain
\begin{align}
\Vert S_{1} \Vert \lesssim& \lambda^{1/2} \int_{0}^{+ \infty} s^{1/2} (s + \lambda^{-1} )^{-1} \lambda^{-1/2 + \varepsilon} \lambda^{1/2 + \varepsilon} (\lambda s +1)^{-1} d s  \nonumber  \\
\lesssim& \lambda^{3/2 + 2 \varepsilon} \int_{0}^{+ \infty} s^{1/2} ( \lambda s +1)^{-2} d s \lesssim \lambda^{2 \varepsilon} \int_{0}^{+ \infty} t^{1/2} ( t +1)^{-2} d t \lesssim \lambda^{2 \varepsilon} , \label{b74}
\end{align}
for all $\varepsilon >0$.

$\star$ We now study $S_{2} = \varphi ( \lambda P) [ A_{0} , R ] \varphi ( \lambda P )$. Using \eqref{b72}, $S_{2}$ can be decomposed as
\begin{equation} \label{b75}
S_{2} = T_{1} + T_{2} + T_{3} ,
\end{equation}
with
\begin{align*}
T_{1} =& \frac{2}{\pi^{2}} \lambda^{5/2} \iint_{0}^{+ \infty} s^{1/2} \overline{\partial} \widetilde{\varphi} ( z ) ( \lambda s + z)^{-1} \varphi ( \lambda P ) ( \lambda P - z)^{-1}  [ P , A_{0} ] ( \lambda P - z)^{-1} \\
&\hspace{160pt} \big( \widetilde{\partial}^{*} r_{0} \widetilde{\partial} + \widetilde{\partial}^{*} r_{1} + r_{1} \widetilde{\partial} \big) (s+P)^{-1} \varphi^{2} ( \lambda P) \, d s \, L (d z ) ,  \\
T_{2} =& \frac{2}{\pi} \lambda^{1/2} \int_{0}^{+ \infty} s^{1/2} \varphi^{2} ( \lambda P) (s+P)^{-1} \big[ A_{0} , \widetilde{\partial}^{*} r_{0} \widetilde{\partial} + \widetilde{\partial}^{*} r_{1} + r_{1} \widetilde{\partial} \big]  (s+P)^{-1} \varphi^{2} ( \lambda P ) \, d s , \\
T_{3} =& \frac{2}{\pi^{2}} \lambda^{5/2} \iint_{0}^{+ \infty} s^{1/2} \overline{\partial} \widetilde{\varphi} ( z ) ( \lambda s +z )^{-1} \varphi^{2} ( \lambda P) (s+P)^{-1} \big( \widetilde{\partial}^{*} r_{0} \widetilde{\partial} + \widetilde{\partial}^{*} r_{1} + r_{1} \widetilde{\partial} \big)    \\
&\hspace{175pt} ( \lambda P - z)^{-1} [ P , A_{0} ] ( \lambda P - z)^{-1} \varphi ( \lambda P ) \, d s \, L (d z ) .
\end{align*}
Since $T_{3}^{*} = T_{1}$, we only treat $T_{1}$ and $T_{2}$.

From \eqref{c7}, we know that
\begin{equation*}
[P , A_{0} ] = - 2 i P + \widetilde{\partial}^{*} r_{0} \widetilde{\partial} + \widetilde{\partial}^{*} r_{1} + r_{1} \widetilde{\partial} .
\end{equation*}
Let $\widetilde{r}$ be a term of the last equation and let $\widehat{r}$ be a term of the sum
\begin{equation*}
\widetilde{\partial}^{*} r_{0} \widetilde{\partial} + \widetilde{\partial}^{*} r_{1} + r_{1} \widetilde{\partial} .
\end{equation*}
Then the functional calculus, Proposition \ref{b45} (with $\beta =0$ and $\gamma =1/2$), Lemma \ref{b20} (with $\beta =0$ and $\gamma =1/2$), Remark \ref{b59} and Lemma \ref{b53} show that
\begin{equation*}
\big\Vert ( \lambda P -z )^{-1} \widetilde{r} ( \lambda P -z )^{-1} \widehat{r} ( \lambda P +1 )^{-1} \big\Vert \lesssim \frac{\lambda^{-2 + \varepsilon}}{\vert \im  z \vert^{C}} ,
\end{equation*}
for all $\varepsilon >0$. Then, $T_{1}$ becomes
\begin{equation} \label{b76}
\Vert T_{1} \Vert \lesssim \lambda^{5/2} \int_{0}^{+ \infty} s^{1/2} (1 + \lambda s)^{-1} \lambda^{-2 + \varepsilon} (s + \lambda^{-1} )^{-1} d s \lesssim \lambda^{\varepsilon} ,
\end{equation}
for all $\varepsilon >0$.

A direct calculation shows that
\begin{equation*}
\big[ A_{0} , \widetilde{\partial}^{*} r_{0} \widetilde{\partial} + \widetilde{\partial}^{*} r_{1} + r_{1} \widetilde{\partial} \big] = \widetilde{\partial}^{*} r_{0} \widetilde{\partial} + \widetilde{\partial}^{*} r_{1} + r_{1} \widetilde{\partial} + r_{2} .
\end{equation*}
From Proposition \ref{b45} (with $\beta =0$ and $\gamma =1/2$) and Lemma \ref{b53}, every term (say $\widetilde{r}$) of the previous equation satisfies
\begin{equation*}
\big\Vert ( \lambda P +1)^{-1} \widetilde{r} ( \lambda P +1)^{-1} \big\Vert \lesssim \lambda^{-1 + \varepsilon} .
\end{equation*}
Then, using the spectral theorem, $T_{2}$ fulfills
\begin{equation} \label{b77}
\Vert T_{2} \Vert \lesssim \lambda^{1/2} \int_{0}^{+ \infty} s^{1/2} ( s + \lambda^{-1})^{-1} \lambda^{-1 + \varepsilon} (s + \lambda^{-1} )^{-1} d s \lesssim \lambda^{\varepsilon} ,
\end{equation}

Combining \eqref{b75} with the estimates \eqref{b76}, \eqref{b77} and $T_{3}^{*} = T_{1}$, we obtain
\begin{equation} \label{b78}
\Vert S_{2} \Vert \lesssim \lambda^{\varepsilon} .
\end{equation}

$\star$ The lemma follows from \eqref{b65}, \eqref{b74}, \eqref{b78} and $S_{3} = - S_{1}^{*}$.
\end{proof}

\begin{lemma}\sl \label{LM4}
Let $0 \leq \mu \leq 1$ and $\psi \in C^{\infty}_{0} ( ] 0 , + \infty [ )$. Then, we have
\begin{gather}
\big\Vert \vert {\mathcal A}_{\lambda} \vert^{\mu} \< x \>^{- \mu} \big\Vert \lesssim \lambda^{- \mu /2 + \varepsilon} ,   \label{M11}  \\
\big\Vert \< {\mathcal A}_{\lambda} \>^{\mu} \psi ( \lambda P) \< x \>^{- \mu} \big\Vert \lesssim \lambda^{- \mu /2 + \varepsilon} ,    \label{M12}
\end{gather}
for all $\varepsilon >0$.
\end{lemma}

\begin{proof}
From \eqref{b29}, we have
\begin{equation} \label{b30}
{\mathcal A}_{\lambda} = \varphi ( \lambda P ) \Big( g^{-1} D x g + i \Big( \frac{d}{2} + g^{-1} x ( \partial g ) \Big) \Big) \varphi ( \lambda P) .
\end{equation}
Lemma \ref{b53} gives
\begin{equation*}
\varphi ( \lambda P ) g^{-1} D = P^{1/2} \varphi ( \lambda P ) P^{-1/2} g^{-1} D = \CO ( \lambda^{-1/2} ).
\end{equation*}
Moreover, Lemma \ref{b61} (with $\beta =1$ and $\gamma =0$) implies
\begin{equation*}
\big\Vert \< x \> \varphi ( \lambda P) u \big\Vert \lesssim \lambda^{\varepsilon} \Vert \< x \> u \Vert ,
\end{equation*}
for all $\varepsilon >0$. Summing up the previous estimates, we get
\begin{equation} \label{b31}
\big\Vert \varphi ( \lambda P ) g^{-1} D x g \varphi ( \lambda P) u \big\Vert \lesssim \lambda^{- 1/2 + \varepsilon} \Vert \< x \> u \Vert ,
\end{equation}
for all $\varepsilon >0$. Using Lemma \ref{b61} (with $\beta =0$ and $\gamma = 1/2$) and that $x ( \partial g )$ is bounded by \eqref{c1}, we obtain
\begin{equation}  \label{b32}
\Big\Vert \varphi ( \lambda P) \Big( \frac{d}{2} + g^{-1} x ( \partial g ) \Big) \varphi ( \lambda P) u \Big\Vert \lesssim \big\Vert \varphi ( \lambda P )u \big\Vert \lesssim \lambda^{- 1/2 + \varepsilon} \Vert \< x \> u \Vert .
\end{equation}
The inequality \eqref{M11} follows from \eqref{b30}, \eqref{b31} and \eqref{b32} for $\mu=1$ and for $0\leq \mu \leq 1$ by interpolation. To prove \eqref{M12}, we write
\begin{equation*}
\big\Vert \< {\mathcal A}_{\lambda} \>^{\mu} \psi ( \lambda P) \< x \>^{- \mu} \big\Vert \lesssim \big\Vert \psi ( \lambda P) \< x \>^{- \mu} \big\Vert + \big\Vert \vert {\mathcal A}_{\lambda} \vert^{\mu} \psi ( \lambda P) \< x \>^{- \mu} \big\Vert \lesssim \lambda^{- \mu /2 + \varepsilon} ,
\end{equation*}
where we have again used Lemma \ref{b61} with $\beta =0$ and $\gamma = 1/2$.
\end{proof}

\begin{lemma}\sl  \label{LM7}
For $\lambda$ large enough, we have
\begin{equation*}
\one_{I} ( \lambda P) \big[ i ( \lambda P)^{1/2} , {\mathcal A}_{\lambda} \big] \one_{I} ( \lambda P) \geq \frac{\delta^{2} \sqrt{\inf I}}{2} \one_{I} ( \lambda P ) .
\end{equation*}
\end{lemma}

\begin{proof}
Recall that \eqref{M7.1} gives
\begin{equation*}
\big[ i ( \lambda P)^{1/2} , {\mathcal A}_{\lambda} \big] = ( \lambda P) ^{1/2} \varphi^{2} ( \lambda P) + R .
\end{equation*}
On the other hand, we know by Lemma \ref{LM3} that $\Vert Ru \Vert \lesssim \lambda^{- \widetilde{\varepsilon}} \Vert u \Vert$ for some $\widetilde{\varepsilon} >0$. Using $\varphi (x) > \delta >0$ on $I$ and taking $\lambda$ large enough, we get the lemma.
\end{proof}

\Subsection{Intermediate frequency Mourre estimate}  \label{secM2}

Here, we obtain a Mourre estimate for frequencies inside the compact $[ 1/C , C ]$. For that, we will use a standard argument in scattering theory. Mimicking Section \ref{secM1}, we set
\begin{equation}
{\mathcal A} = \varphi (P) A_{0} \varphi (P) ,
\end{equation}
where $\varphi \in C^{\infty}_{0} ( ] 0 , + \infty [ ; [0, + \infty [)$ with $\varphi =1$ near $[ 1/C , C ]$. As before, ${\mathcal A}$ is essentially self-adjoint on $D(A_0)$ and we denote again ${\mathcal A}$ its closure.

\begin{proposition}\sl  \label{PM2}
$i)$ We have $P^{1/2}\in C^2({\mathcal A})$. The commutators $\ad^j_{{\mathcal A}} P^{1/2}$, $j=1,2$, can be extended to bounded operators.

$ii)$ For each $\sigma \in [ 1/C , C ]$, there exists $\delta >0$ such that
\begin{equation} \label{c12}
\one_{[ \sigma - \delta , \sigma + \delta ]} (P) \big[ i P^{1/2} , {\mathcal A} \big] \one_{[ \sigma - \delta , \sigma + \delta ]} (P) \geq \frac{1}{2 \sqrt{C}} \one_{[ \sigma - \delta , \sigma + \delta ]} (P) .
\end{equation}

$iii)$ For $0 \leq \mu \leq 1$, we have
\begin{equation}
\big\Vert \< {\mathcal A} \>^{\mu} \< x \>^{- \mu} \big\Vert \lesssim 1 .
\end{equation}
\end{proposition}

\begin{proof}
The points $i)$ and $iii)$ follow directly from Proposition \ref{PM1} with $\lambda =1$. Moreover, using \eqref{M7.1} and Lemma \ref{LM3}, we get
\begin{equation*}
\big[ i P^{1/2} , {\mathcal A} \big] = P^{1/2} \varphi^{2} (P) + R ,
\end{equation*}
where $\Vert R u \Vert \lesssim \Vert \< x \>^{- \nu} u \Vert$ for some $\nu >0$. Then, $R \varphi (P)$ is a compact operator on $L^{2} ( \R^{d} )$. Let $\sigma \in [ 1/C , C ]$. Since $\sigma$ is not an eigenvalue of $P$ (see \cite[Corollary 5.4]{Do97_01}), we have
\begin{equation*}
\slim_{\delta \to 0} \one_{[ \sigma - \delta , \sigma + \delta ]} (P) =0 .
\end{equation*}
Thus, we obtain
\begin{equation*}
\lim_{\delta \to 0} \one_{[ \sigma - \delta , \sigma + \delta ]} (P) R \varphi (P) =0 ,
\end{equation*}
in operator norm. Using
\begin{align*}
\one_{[ \sigma - \delta , \sigma + \delta ]} (P) \big[ i P^{1/2} , {\mathcal A} \big] \one_{[ \sigma - \delta , \sigma + \delta ]} (P) \geq & \sqrt{\sigma - \delta} \one_{[ \sigma - \delta , \sigma + \delta ]} (P)   \\
&+ \one_{[ \sigma - \delta , \sigma + \delta ]} (P) R \varphi (P) \one_{[ \sigma - \delta , \sigma + \delta ]} (P) ,
\end{align*}
part $ii)$ of the proposition follows.
\end{proof}

\Subsection{High frequency Mourre estimate}

In this section, we construct a conjugate operator at high frequencies. We work with the simple $\sigma$--temperate metric
\begin{equation*}
\gamma = \frac{d x^{2}}{1 + x^{2}} + \frac{d \xi^{2}}{1 + \xi^{2}} .
\end{equation*}
We refer to \cite[Section XVIII]{Ho85_01} for the Weyl calculus of H\"{o}rmander. For $m (x , \xi )$ a weight function, let $S (m)$ be the set of functions $f \in C^{\infty} ( \R^{d} \times \R^{d})$ such that
\begin{equation*}
\vert \partial_{x}^{\alpha} \partial_{\xi}^{\beta} f (x, \xi ) \vert \lesssim m (x , \xi ) \< x \>^{- \vert \alpha \vert} \< \xi \>^{- \vert \beta \vert} ,
\end{equation*}
for all $\alpha , \beta \in \N^{d}$. In fact, $S (m)$ is the space of symbols of weight $m$ for the metric $\gamma$. Let $\Psi (m)$ denote the set of pseudo-differential operators whose symbols are in $S (m)$.

Let $p (x, \xi ) \in S ( \< \xi \>^{2} )$ be the symbol of $P$, and
\begin{equation*}
p_{0} (x, \xi ) = \sum_{j , k} g^{j , k} (x) \xi_{j} \xi_{k} \in S ( \< \xi \>^{2} ) ,
\end{equation*}
be its principal part. We have $p - p_{0} \in S ( 1)$. Let
\begin{equation*}
{\rm H}_{p_{0}} = \left( \begin{array}{c}
\partial_{\xi} p_{0} \\
- \partial_{x} p_{0}
\end{array} \right) ,
\end{equation*}
be the Hamiltonian of $p_{0}$. Since the metric $\mathfrak{g}$ is non-trapping by assumption, the energy $\{ p_{0} =1 \}$ is non-trapping for the Hamiltonian flow of $p_{0}$. Then, using a result of C. G\'erard and Martinez \cite{GeMa88_01}, one can construct a function $b (x, \xi ) \in S ( \< x \> \< \xi \> )$ such that $b = x \cdot \xi$ for $x$ large enough, and
\begin{equation} \label{a9}
H_{p_{0}} b \geq \delta ,
\end{equation}
for some $\delta >0$ and all $(x, \xi ) \in p_{0}^{-1} ( [ 1 - \varepsilon , 1 + \varepsilon ])$, $\varepsilon >0$. We set $A = \Op (a)$ with
\begin{equation*}
a (x, \xi ) = b \big( x , (p_{0} +1)^{-1/2} \xi \big) \in S ( \< x \> ) .
\end{equation*}
Let $f \in C^{\infty} ( \R ; \R)$ be such that $f =1$ on $[2 , + \infty [$ and $f=0$ on $] - \infty , 1]$. As conjugate operator at high frequency, we choose
\begin{equation}
A_{\infty} = f (P) A f (P) .
\end{equation}

Let $\varphi \in C^{\infty}_{0} ( \R )$ satisfy $\varphi + f =1$ on $[-1 , + \infty [$. Since $P \geq 0$, we have $f (P) = 1 - \varphi ( P)$. On the other hand, from the functional calculus of pseudo-differential operators, $\varphi (P) \in \Psi ( \< \xi \>^{- \infty} )$ and then $f(P) \in \Psi ( 1 )$. To prove this assertion, one can, for instance, adapt Theorem 8.7 of \cite{DiSj99_01} or Section D.11 of \cite{DeGe97_01} to the metric $\gamma$. In particular, $A_{\infty}$ is well defined as a pseudo-differential operator, and we have
\begin{equation} \label{a28}
A_{\infty} = A  + \Psi ( \< x \> \< \xi \>^{- \infty} ) \in \Psi ( \< x \> ).
\end{equation}
The following proposition summarizes the useful properties of $A_{\infty}$.

\begin{proposition}\sl \label{a26}
$i)$ The operator $A_{\infty}$ is essentially self-adjoint on any core of $\< x \>$ with $D ( \< x \> ) = \{ u \in L^{2} ( \R^{d}) ; \ \< x \> u \in L^{2} ( \R^{d} ) \}$. Moreover,
\begin{equation*}
\Vert  A_{\infty} u \Vert \lesssim \Vert \< x \> u \Vert ,
\end{equation*}
for all $u \in D ( \< x \> )$.

$ii)$ We have $P^{1/2} \in C^{2} (A_{\infty})$. The commutators $[P^{1/2} ,A_{\infty}]$ and $[ [P^{1/2} ,A_{\infty}] ,A_{\infty} ]$ are in $\Psi (1)$ and can be extended as bounded operators on $L^{2} ( \R^{d} )$.

$iii)$ For $C >0$ large enough,
\begin{equation*}
\one_{[C , + \infty [} (P) i [ P^{1/2} , A_{\infty} ] \one_{[C , + \infty [} (P) \geq \frac{\delta}{8} \one_{[C , + \infty [} (P) .
\end{equation*}
\end{proposition}

The rest of this subsection is devoted to the proof of this proposition. It is a direct consequence of the next lemmas. For the first part of the proposition, we will use the following extension of Nelson's theorem due to C. G\'erard and {\L}aba \cite[Lemma 1.2.5]{GeLa02_01} (see also Reed and Simon \cite[Theorem~X.36]{ReSi75_01}).

\begin{theorem}[Nelson's theorem]\sl \label{a12}
Let ${\mathcal H}$  be a Hilbert space, $N\geq 1$ a self-adjoint operator on ${\mathcal H}$, $H$ a symmetric operator such that $D(N) \subset D(H)$ and 
\begin{gather*}
\Vert H u \Vert \lesssim \Vert N u \Vert ,    \\
\vert (H u,N u) - (N u,H u) \vert \lesssim \Vert N^{1/2} u \Vert^{2} ,
\end{gather*}
for all $u \in D ( N )$. Then, $H$ is essentially self-adjoint on any core of $N$.
\end{theorem}

\begin{lemma}\sl \label{a14}
The operator $A_{\infty}$ is essentially self-adjoint on any core of $(\< x \> , D ( \< x \> ))$ (in particular, on the Schwartz space ${\mathcal S} ( \R^{d} )$) and
\begin{equation*}
\Vert  A_{\infty} u \Vert \lesssim \Vert \< x \> u \Vert ,
\end{equation*}
for all $u \in D ( \< x \> )$.
\end{lemma}

\begin{proof}
The operator $N = \< x \>$ is self-adjoint on $D (N) = D ( \<  x\> )$ and essentially self-adjoint on ${\mathcal S} ( \R^{d} )$. Since $A_{\infty} \in \Psi ( \< x \> )$ and $N^{-1} \in \Psi ( \< x \>^{-1} )$, the operator $A_{\infty} N^{-1} \in \Psi (1)$ is bounded on $L^{2} ( \R^{d} )$ by Calderon and Vaillancourt's theorem. Then, $A_{\infty}$ is defined on $D ( N)$ and
\begin{equation*}
\Vert A_{\infty} u \Vert \lesssim \Vert N u \Vert ,
\end{equation*}
for all $u \in D(N)$.

By pseudo-differential calculus, $\< x \>^{-1/2} [ A_{\infty} , \< x \> ] \< x \>^{-1/2} \in \Psi ( \< \xi \>^{-1} )$ is bounded on $L^{2} ( \R^{d} )$. Then, working first on ${\mathcal S} ( \R^{d} )$, this gives
\begin{equation*}
\vert (A u,N u) - (N u,A u) \vert \lesssim \Vert N^{1/2} u \Vert^{2} ,
\end{equation*}
for all $u \in D(N)$. Thus, Theorem \ref{a12} implies that $A_{\infty}$ is essentially self-adjoint on any core of $D ( \< x \> )$.
\end{proof}

\begin{lemma}\sl \label{a13}
Let $g \in C^{\infty} ( \R ; [0 , + \infty [ )$ be such that $g = 0$ on $] - \infty , a ]$ and $g =1$ on $[ b , + \infty [$, for some $0 < a < b$. Then,
\begin{equation*}
g (P) P^{1/2} = \Op \big( ( p_{0} + 1 )^{1/2} \big) + \Psi ( 1 ) \in \Psi ( \< \xi \> ) .
\end{equation*}
\end{lemma}

We omit the proof of this classical result. It follows from \eqref{c8} and the Beals lemma \cite{Be81_01}. We refer to Section 4.4 of \cite{HeNi05_01} for similar arguments (see also \cite[Theorem 4.9]{Be81_01}).

\begin{remark}\sl
For the subsequent uses, a parametrix will be enough. In fact, since we work with the metric $\gamma$, the remainder terms will decay like $\< ( x , \xi ) \>^{- \infty}$. Therefore, they can ``absorb'' the pseudo-differential operators of any weight. In particular, this allows to treat the commutators.
\end{remark}

\begin{lemma}\sl \label{a18}
We have $[P^{1/2} ,A_{\infty}] \in \Psi (1)$ and $[ [P^{1/2} ,A_{\infty}] ,A_{\infty} ] \in \Psi ( \< \xi \>^ {-1})$. These commutators extend as bounded operators on $L^{2} ( \R^{d} )$.
\end{lemma}

\begin{proof}
Let $g \in C^{\infty} ( \R )$ as in Lemma \ref{a13} be such that $f g = f$. Then,
\begin{equation} \label{a27}
\big[  P^{1/2} ,A_{\infty} \big] = \big[ g (P) P^{1/2} ,A_{\infty} \big] .
\end{equation}
Since $g (P) P^{1/2} \in \Psi ( \< \xi \> )$ by Lemma \ref{a13} and $A_{\infty} \in \Psi ( \< x \> )$, the pseudo-differential calculus gives $[ g (P) P^{1/2} ,A_{\infty} ] \in \Psi (1)$. The same way, $[ [P^{1/2} ,A_{\infty}] ,A_{\infty} ] \in \Psi ( \< \xi \>^ {-1})$. Using Calderon and Vaillancourt's theorem and working first on ${\mathcal S} ( \R^{d})$ which is dense in $D ( \< x \> ) \cap H^{1} ( \R^{d} )$, one can prove that these operators extend as bounded operator on $L^{2} ( \R^{d} )$.
\end{proof}

\begin{lemma}\sl
We have $P^{1/2} \in C^{2} ( A_{\infty} )$.
\end{lemma}

\begin{proof}
Let $H = \< D \> = \Op ( \< \xi \> ) \in \Psi ( \< \xi \> )$ be the self-adjoint operator with domain $D (H) = D (P^{1/2} ) = H^{1} ( \R^{d} )$ (see Lemma \ref{b53}). We remark that $(H \pm z)^{-1} = \Op ( (\< \xi \> \pm z)^{-1} ) \in \Psi (1)$ is a Fourier multiplier. Thus, $D ( \< x \> )$, which is a core of $A_{\infty}$ from Lemma \ref{a14}, is stable by $(H \pm z)^{-1}$. On the other hand, $[ H , A_{\infty}] \in \Psi (1)$ can be extend as a bounded operator on $L^{2} ( \R^{d} )$. Then, Theorem \ref{a15} implies $H \in C^{1} ( A_{\infty} )$.

Since $H \in C^{1} ( A_{\infty} )$ and $[ H , A_{\infty}]$ is bounded on $L^{2} ( \R^{d})$, Lemma \ref{a17} says that $e^{i t A_{\infty}}$ leaves $D (H)$ invariant. Then, $e^{i t A_{\infty}}$ leaves $D (P^{1/2}) = D (H)$ invariant and $[ P^{1/2} , A_{\infty}]$ is bounded from Lemma \ref{a18}. Then, Theorem \ref{a16} implies that $P^{1/2} \in C^{1} ( A_{\infty} )$.

The lemma follows from Theorem \ref{a16}, Remark \ref{a7} and Lemma \ref{a18}.
\end{proof}

\begin{lemma}\sl
For $C >0$ large enough,
\begin{equation*}
\one_{[C , + \infty [} (P) i [ P^{1/2} , A_{\infty} ] \one_{[C , + \infty [} (P) \geq \frac{\delta}{8} \one_{[C , + \infty [} (P) .
\end{equation*}
\end{lemma}

\begin{proof}
Equation \eqref{a28} and \eqref{a27}, Lemma \ref{a13} and the pseudo-differential calculus give
\begin{align}
i \big[ P^{1/2} , A_{\infty} \big] =& i \big[ g (P) P^{1/2} , A_{\infty} \big] \nonumber  \\
=& i \big[ \Op \big( ( p_{0} + 1)^{1/2} \big) , \Op ( a) \big] + \Psi ( \< \xi \>^{-1} )   \nonumber \\
=& \frac{1}{2} \Op \big( ( p_{0} + 1)^{-1/2} {\rm H}_{p_{0}} a \big) + \Psi ( \< \xi \>^{-1} )    \nonumber \\
=& \frac{1}{2} \Op \Big( (p_{0} +1)^{-1/2} (\partial_{\xi} p_{0}) (x, \xi ) \cdot ( \partial_{x} b ) \big( x , (p_{0} +1)^{-1/2} \xi \big)   \nonumber  \\
&- (p_{0} +1)^{-1} (\partial_{x} p_{0}) (x, \xi ) \cdot ( \partial_{\xi} b ) \big( x , (p_{0} +1)^{-1/2} \xi \big) \Big) + \Psi ( \< \xi \>^{-1} )   \nonumber \\
=& \frac{1}{2} \Op \big( ( {\rm H}_{p_{0}} b ) \big( x , (p_{0} +1)^{-1/2} \xi \big) \big) + \Psi ( \< \xi \>^{-1} ) . \label{a8}
\end{align}
For the last equality, we have used that $p_{0}$ is a homogeneous polynomial of order $2$ in $\xi$.

Note that
\begin{equation*}
p_{0} \big( x , (p_{0} +1)^{-1/2} \xi \big) = (p_{0} +1)^{-1} p_{0} \in [1 - \varepsilon , 1 + \varepsilon ],
\end{equation*}
for $\xi$ large enough. Then, adding a cut-off function in $\xi$, \eqref{a9} and \eqref{a8} imply that
\begin{equation*}
i \big[ P^{1/2} , A_{\infty} \big] = \Op ( c ) + \Psi ( \< \xi \>^{-1} ) ,
\end{equation*}
with $c \in S (1)$ and $c (x, \xi ) \geq \delta /2$. We write $c (x , \xi ) = \delta / 4 + d^{2} (x, \xi )$ with $d \in S (1)$ real valued. Thus, by the pseudo-differential calculus,
\begin{align}
i \big[ P^{1/2} , A_{\infty} \big] &\geq \delta / 4 + \Op ( d)^{*}  \Op ( d) + \Psi ( \< \xi \>^{-1} )  \nonumber \\
&\geq \delta / 4  + \Psi ( \< \xi \>^{-1} ) , \label{a10}
\end{align}
as self-adjoint operators (one can also apply the G{\aa}rding inequality).

Let $R \in \Psi ( \< \xi \>^{-1} )$. Since $P \in \Psi ( \< \xi \>^{2})$, the operator
\begin{equation*}
R^{*} ( P +1 ) R \in \Psi (1) ,
\end{equation*}
is a bounded operator on $L^{2} ( \R^{d} )$. Then, $( P +1)^{1/2} R$ is also bounded on $L^{2} ( \R^{d} )$. In particular, we have
\begin{align}
\big\Vert \one_{[C , + \infty [} (P) R \big\Vert =& \big\Vert \one_{[C , + \infty [} (P) ( P +1)^{-1/2} ( P +1)^{1/2} R \big\Vert    \nonumber  \\
\lesssim& \big\Vert \one_{[C , + \infty [} (P) ( P +1)^{-1/2} \big\Vert \big\Vert ( P +1)^{1/2} R \big\Vert    \nonumber \\
\lesssim& ( C +1)^{-1/2} . \label{a11}
\end{align}
The lemma follows from \eqref{a10} and \eqref{a11}.
\end{proof}

\section{Proof of the linear estimates} \label{EW}

In this section, we will show the main estimates for the linear wave equation (Theorem~\ref{TSLW2} and Theorem~\ref{TW1}). To prove these results, we will make a dyadic decomposition of the low frequencies. We will often consider $\varphi \in C^{\infty}_{0} ( ] 0 , + \infty [ ; [0, + \infty [)$ such that
\begin{equation} \label{c13}
\sum_{\lambda = 2^{n} , \ n \geq 0} \varphi ( \lambda x) = 1 ,
\end{equation}
for $x \in ] 0 , 1 ]$. To $\varphi$, we will associate $\widetilde{\varphi} \in C^{\infty}_{0} ( ] 0 , + \infty [  ; [0, + \infty [)$ satisfying $\widetilde{\varphi} \varphi = \varphi$.

We begin with a technical lemma which proves Remark \ref{RW1} $ii)$.

\begin{lemma}\sl \label{LW5}
For all $\widetilde{\mu} < \mu \leq 3/2$, we have
\begin{gather*}
\big\Vert \<x\>^{-\mu} \widetilde{\partial}_{\ell} u \big\Vert \lesssim \big\Vert \<x\>^{- \widetilde{\mu}} P^{1/2} u \big\Vert ,  \\
\big\Vert \<x\>^{- \mu} P^{1/2} u \big\Vert \lesssim \sum_{\ell =1}^{d} \big\Vert \<x\>^{- \widetilde{\mu}} \widetilde{\partial}_{\ell} u \big\Vert .
\end{gather*}
\end{lemma}

\begin{proof}
Since the two inequalities can be treated the same way, we only prove the first one. We write
\begin{equation*}
\big\Vert \<x\>^{-\mu} \widetilde{\partial}_{\ell} u \big\Vert \leq \big\Vert \<x\>^{-\mu} \widetilde{\partial}_{\ell} \Psi (P \leq C) u \big\Vert + \big\Vert \<x\>^{-\mu} \widetilde{\partial}_{\ell} \Psi (P \geq C) u \big\Vert =: I_1 + I_2 .
\end{equation*}

$\bullet$ We first estimate $I_1$. Let $\varphi$ be as in \eqref{c13}. For $\widetilde{\mu} < \mu$, we have, using Lemma \ref{b61},
\begin{align*}
I_1 \lesssim& \sum_{\lambda \text{ dyadic}} \big\Vert \< x \>^{- \mu} \widetilde{\partial}_{\ell} \varphi (\lambda P) u \big\Vert  \\
=& \sum_{\lambda \text{ dyadic}} \big\Vert \< x \>^{- \mu} ( \lambda^{1/2} \widetilde{\partial}_{\ell} ) \varphi ( \lambda P) ( \lambda P)^{-1/2} \< x \>^{\widetilde{\mu}} \< x \>^{- \widetilde{\mu}} P^{1/2}u \big\Vert  \\
\lesssim& \sum_{\lambda \text{ dyadic}} \lambda^{-\frac{\mu-\widetilde{\mu}}{2} +\varepsilon} \big\Vert \< x \>^{-\widetilde{\mu}} P^{1/2} u \big\Vert  \\
\lesssim& \big\Vert \< x \>^{-\widetilde{\mu}} P^{1/2} u \big\Vert ,
\end{align*}
for $\varepsilon$ small enough.

$\bullet$ We now estimate $I_{2}$. By Lemma \ref{a13} and the pseudo-differential calculus, we know that the operator
\begin{equation*}
\<x\>^{-\mu} \widetilde{\partial}_{\ell} \Psi (P \geq C) P^{-1/2} \<x\>^{\mu} ,
\end{equation*}
is bounded. Therefore,
\begin{equation*}
I_{2} = \big\Vert \<x\>^{-\mu} \widetilde{\partial}_{\ell} \Psi (P \geq C) P^{-1/2} \< x \>^{\mu} \< x \>^{- \mu} P^{1/2} u \big\Vert \lesssim \big\Vert \< x \>^{- \mu} P^{1/2} u \big\Vert .
\end{equation*}
\end{proof}

Using the same type of proof, one can show the following estimate.

\begin{lemma}\sl \label{c16}
For all $\mu >1$, we have
\begin{equation*}
\big\Vert \< x \>^{- \mu} u \big\Vert \lesssim \big\Vert P^{1/2} u \big\Vert .
\end{equation*}
\end{lemma}

\begin{remark}\sl \label{RW2}
Let $\mu>0$ be given. Then, for all $\varepsilon >0$, there exist $0< \widetilde{\mu} < \mu$, $0 < \widetilde{\varepsilon} < \varepsilon$ such that $F^{\widetilde{\varepsilon}}_{\widetilde{\mu}} (T) \leq F^{\varepsilon}_{\mu} (T)$. Then, it is sufficient to bound the different quantities we consider by $F^{\varepsilon}_{\widetilde{\mu}}(T)$ rather than by $F^{\varepsilon}_{\mu} (T)$.
\end{remark}

Theorem \ref{TSLW2} will follow from the corresponding result for the group $e^{-i t P^{1/2}}$.

\begin{proposition}\sl \label{propEW2}
Let $0 < \mu \leq 1$. Then, for all $0 < \varepsilon < \mu$, we have
\begin{equation*} \label{EW3}
\int_0^T \big\Vert \<x\>^{- \mu} e^{-i t P^{1/2}} v \big\Vert^2 d t \lesssim F_{\mu}^{\varepsilon} (T) \Vert v \Vert^2 .
\end{equation*}
\end{proposition}

\begin{proof}
We write
\begin{align*}
\int_0^T \big\Vert \< x \>^{- \mu} e^{-i t P^{1/2}} v \big\Vert^2 d t \lesssim & \int_0^T \big\Vert \< x \>^{- \mu} e^{-i t P^{1/2}} \Psi (P \leq 1/C) v \big\Vert^2 d t    \\
&+ \int_0^T \big\Vert \< x \>^{- \mu} e^{-i t P^{1/2}} \Psi ( 1/C \leq P \leq C) v \big\Vert^2 d t   \\
&+ \int_0^T \big\Vert \< x \>^{- \mu} e^{-i t P^{1/2}} \Psi ( P \geq C ) v \big\Vert^2 d t =: I_{1} + I_{2} + I_{3} .
\end{align*}

$\bullet$ We first estimate $I_1$. Let $\varphi$, $\widetilde{\varphi}$ be as in \eqref{c13}. Proposition \ref{PM1} gives
\begin{equation*}
\big\Vert \< x \>^{- \mu} \widetilde{\varphi} ( \lambda P) \< {\mathcal A}_{\lambda} \>^{\mu} \big\Vert^2 \lesssim \lambda^{- \mu + \varepsilon_1} ,
\end{equation*}
for all $\varepsilon_{1} >0$. Then,
\begin{align*}
I_1 \lesssim& \sum_{\lambda \text{ dyadic}}\lambda^{-\mu+\varepsilon_1}\int_0^T \big\Vert \< {\mathcal A}_{\lambda}\>^{-\mu}e^{-i t P^{1/2}} \varphi ( \lambda P) v \big\Vert^2 d t  \\
=& \sum_{\lambda \text{ dyadic}} \lambda^{- \mu + \varepsilon_1 + 1/2} \int_0^{\lambda^{-1/2} T} \big\Vert \< {\mathcal A}_{\lambda} \>^{- \mu} e^{-i s (\lambda P)^{1/2}}\varphi ( \lambda P) v \big\Vert^2 d s   \\
\lesssim& \sum_{\lambda \text{ dyadic}} \lambda^{- \mu + \varepsilon_1 +1/2+ \varepsilon_3} F_{\mu}^{\varepsilon_2} (\lambda^{-1/2} T)\Vert v \Vert^2,
\end{align*}
for all $\varepsilon_{2}, \varepsilon_{3} >0$ with $\varepsilon_{2} < \mu$. Here, we have used Proposition \ref{PM1}, Remark \ref{a6}, Theorem \ref{a1} (for $\mu > 1/2$) and Corollary \ref{a2} (for $\mu \leq 1/2$) with $H = ( \lambda P )^{1/2}$.
\begin{itemize}
\item[$\star$]
If $\mu>1/2$, then, by choosing $\varepsilon_1 , \varepsilon_3$ small enough, the sum is convergent and we find
\begin{equation*}
I_1 \lesssim \Vert v \Vert^2 .
\end{equation*}
\item[$\star$]
If $\mu\le1/2$, we find
\begin{equation*}
I_1 \lesssim \sum_{\lambda \text{ dyadic}} \lambda^{\varepsilon_1 + \varepsilon_3 - \varepsilon_2} T^{1-2\mu+2\varepsilon_2} \Vert v \Vert^2 .
\end{equation*}
Fixing first $\varepsilon_2$ and then $\varepsilon_1 , \varepsilon_3$ small enough makes the sum convergent.
\end{itemize}

$\bullet$ We now treat $I_{2}$. Since $[1/C , C]$ is a compact interval, Proposition \ref{PM2} gives us a finite number of open intervals $I_{k}$, $k =1 , \ldots , K$, satisfying \eqref{c12} and
\begin{equation*}
[1/C , C] \subset \bigcup_{k=1}^{K} I_{k} .
\end{equation*}
Then, applying Theorem \ref{a1} (for $\mu > 1/2$) and Corollary \ref{a2} (for $\mu \leq 1/2$) on each $I_{k}$ (slightly reduced), we obtain
\begin{equation*}
I_{2} \lesssim F_{\mu}^{\varepsilon} (T) \Vert v \Vert^2 .
\end{equation*}

$\bullet$ Let us finally estimate $I_{3}$. By Proposition \ref{a26} and an interpolation argument, we get
\begin{equation*}
\big\Vert \< x \>^{- \mu} \< {\mathcal A}_{\infty} \>^{\mu} \big\Vert \lesssim 1.
\end{equation*}
Thus,
\begin{align*}
\int_0^T \big\Vert \< x \>^{- \mu} e^{-i t P^{1/2}} \Psi (P \geq C) v \big\Vert^2 d t \lesssim& \int_0^T \big\Vert \< {\mathcal A}_{\infty} \>^{- \mu} e^{-i t P^{1/2}} \Psi(P \geq C) v \big\Vert^2 d t   \\
\lesssim& F_{\mu}^{\varepsilon}(T) \Vert v \Vert^2 ,
\end{align*}
where we have used Theorem \ref{a1} (for $\mu > 1/2$) and Corollary \ref{a2} (for $\mu \leq 1/2$).
\end{proof}

For the proof of Theorem \ref{TSLW2}, we will need the following theorem of Christ and Kiselev \cite{ChKi01_01} in a form available in the article of Burq \cite{Bu03_01}.

\begin{theorem}[Christ--Kiselev]\sl
Consider a bounded operator ${\mathcal T} : L^p( \R ;B_1) \longrightarrow L^q ( \R ; B_2)$ given by a locally integrable Kernel $K(t,s)$ with value operators from $B_1$ to $B_2$, where $B_{1}$ and $B_{2}$ are Banach spaces. Suppose that $p<q$. Then, the operator
\begin{equation*}
\widetilde{{\mathcal T}} f (t) = \int_{s<t} K (t,s) f(s) \, d s ,
\end{equation*}
is bounded from $L^p(\R;B_1)$ to $L^q(\R;B_2)$ by
\begin{equation*}
\big\Vert \widetilde{{\mathcal T}} \big\Vert_{L^p(\R;B_1) \rightarrow L^q(\R;B_2)} \leq \big( 1-2^{-p^{-1}-q^{-1}} \big) \Vert {\mathcal T} \Vert_{L^p( \R ;B_1) \rightarrow L^q ( \R ;B_2)}.
\end{equation*}
\end{theorem}

\begin{proof}[Proof of Theorem \ref{TSLW2}]
By linearity and uniqueness it is sufficient to consider separately the cases $(u_0,u_1)=0$, $G=0$.

$\bullet$ $G=0$. Thanks to the discussion at the beginning of Section \ref{secM}, the solution of (\ref{LW}) is given by 
\begin{equation*}
\left(\begin{array}{c}
u(t) \\ \partial_t u(t)
\end{array} \right)
= e^{-i t R}
\Big( \begin{array}{c}
u_0 \\ u_1
\end{array} \Big)
\quad \text{ with } \quad R= \left( \begin{array}{cc}
0 & i
\\ - i P & 0
\end{array} \right), \ R = U^* L U .
\end{equation*}
Using Lemma \ref{LW5}, we see that for $\widetilde{\mu} < \mu$ we have
\begin{equation*}
\big\Vert \< x \>^{- \mu} u' \big\Vert_{L^2}^2 \lesssim \Big\Vert \< x \>^{- \widetilde{\mu}} M
\left( \begin{array}{c}
u(t)  \\
\partial_t u(t)
\end{array} \right)
\bigg\Vert_{L^2\times L^2}^2,
\end{equation*}
with
\begin{equation} \label{c14}
M:= \left( \begin{array}{cc} P^{1/2} & 0 \\ 0 & 1 \end{array} \right) , \  M U^* = \frac{1}{\sqrt{2}} \left( \begin{array}{cc} 1 & 1 \\ -i & i \end{array} \right) .
\end{equation}
Using Proposition \ref{propEW2}, we therefore have the following estimate
\begin{align*}
\int_0^T \big\Vert \< x \>^{- \mu} u' \big\Vert_{L^2}^2 d t \lesssim&\int_0^T \Big\Vert \< x \>^{- \widetilde{\mu}} M e^{-i t R} \Big( \begin{array}{c} u_0 \\ u_1 \end{array} \Big) \Big\Vert^2_{L^2 \times L^2} d t  \\
\lesssim& \int_0^T \Big\Vert \< x \>^{- \widetilde{\mu}} e^{-i t L} U \Big( \begin{array}{c} u_0 \\ u_1 \end{array} \Big) \Big\Vert_{L^2 \times L^2}^2 d t    \\
\lesssim& F^{\varepsilon}_{\widetilde{\mu}} (T) \Big\Vert U \Big( \begin{array}{c} u_0 \\ u_1 \end{array} \Big) \Big\Vert_{L^2 \times L^2}^2 = F^{\varepsilon}_{\widetilde{\mu}} (T) \big\Vert (u_0 , u_1 ) \big\Vert^2_{\mathcal E}.
\end{align*}

$\bullet$ $(u_0,u_1)=0$. In this case, the solution of (\ref{LW}) is given by
\begin{equation*}
\left( \begin{array}{c} u (t) \\ \partial_t u (t) \end{array} \right) = \int_0^t e^{i(s-t)R} \left( \begin{array}{c} 0 \\ G(s) \end{array} \right) d s .
\end{equation*}
Thus, for all $\widetilde{\mu} < \mu$,
\begin{equation} \label{EW4}
\int_0^T \big\Vert \< x \>^{- \mu} u' \Vert^2_{L^2} \lesssim \int_0^T \left\Vert \int_0^t \< x \>^{-\widetilde{\mu}} e^{i(s-t)L} U \left( \begin{array}{c} 0 \\ G (s) \end{array} \right) d s \right\Vert^2_{L^2 \times L^2} d t .
\end{equation}
Let 
\begin{equation*}
{\mathcal T} f (t) = \int_{\R} \< x \>^{- \widetilde{\mu}} \one_{[ 0,T ]} (s) \one_{[ 0,T ]} (t) e^{i(s-t)L} f (s) \, d s .
\end{equation*}
We estimate
\begin{align*}
\Vert {\mathcal T} f \Vert^2_{L^2( \R ; L^2 \times L^2 )} =& \int_0^T \left\Vert \< x \>^{- \widetilde{\mu}} e^{-i t L} \int_0^T e^{i s L} f (s) \, d s \right\Vert^2 d t  \\
\lesssim& F^{\varepsilon}_{\widetilde{\mu}} (T) \left\Vert \int_0^T e^{i s L} f (s) \, d s \right\Vert^2   \\
\lesssim& F^{\varepsilon}_{\widetilde{\mu}} (T) \left(\int_0^T \Vert f (s) \Vert \, d s \right)^2.
\end{align*}
It follows
\begin{equation*}
\Vert {\mathcal T} \Vert_{L^1( \R ;L^2 \times L^2) \rightarrow L^2( \R ;L^2 \times L^2)}^{2} \lesssim F^{\varepsilon}_{\widetilde{\mu}} (T).
\end{equation*}
The expression on the right hand side of (\ref{EW4}) is 
\begin{equation*}
\big\Vert \widetilde{{\mathcal T}} U (0,G(s)) \big\Vert^2_{L^2( \R ;L^2 \times L^2)} ,
\end{equation*}
with
\begin{equation*}
\widetilde{{\mathcal T}} f (t) = \int_{s < t} \< x \>^{- \widetilde{\mu}} \one_{[ 0,T ]} (s) \one_{ [ 0 ,T ]} (t) e^{i (s-t) L} f (s) \, d s .
\end{equation*}
We can apply the theorem of Christ and Kiselev to conclude that 
\begin{align*}
\big\Vert \widetilde{T} U (0,G(s)) \big\Vert^2_{L^2( \R ;L^2 \times L^2)} \lesssim& F^{\varepsilon}_{\widetilde{\mu}} (T) \left( \int_0^T \left\Vert U \left( \begin{array}{c} 0 \\ G(s)\end{array} \right) \right\Vert_{L^2 \times L^2} d s \right)^2  \\
=& F^{\varepsilon}_{\widetilde{\mu}} (T) \left( \int_0^T \left\Vert G (s) \right\Vert_{L^2} d s \right)^2 ,
\end{align*}
which finishes the proof.
\end{proof}

Theorem \ref{TW1} is now proved for $N=0$ using in addition the usual energy estimate
\begin{equation} \label{energy}
\Vert u' \Vert_{L^2 ( \R^{d} )} \lesssim \Vert u'(0, \cdot ) \Vert_{L^2 ( \R^{d} )} + \int_0^T \Vert G(s, \cdot ) \Vert_{L^2 ( \R^{d} )} d s .
\end{equation}
Note that in the usual energy estimate $u'$ is replaced by $(\partial_t u ,P^{1/2}u)$, but we have
\begin{equation*}
\sum_k \big\Vert \widetilde{\partial}_k u \big\Vert \lesssim \Vert P^{1/2} u \Vert \lesssim \sum_k \big\Vert \widetilde{\partial}_k u \big\Vert ,
\end{equation*}
by Lemma \ref{b53}. It will be useful to have similar estimates to the preceding containing a $L^2( \R^{d+1}, \< x \>^{\mu} d t \, d x )$ norm of $G$ on the right hand side rather than a $L^1_t L^2_x$ norm.

\begin{proposition}\sl \label{PW4}
Assume $0<\mu\leq 1$.

$i)$ Let
\begin{equation} \label{W16}
\left\{ \begin{aligned}
&( i \partial_t - P^{1/2} ) v = G ,  \\
&v_{\vert_{t=0}} = 0 .
\end{aligned} \right.
\end{equation}
Then we have, for all $0 < \varepsilon < \mu$,
\begin{equation} \label{W17}
\int_0^T \big\Vert \< x \>^{- \mu} v \big\Vert^2 d t \lesssim ( F^{\varepsilon}_{\mu} (T) )^{2} \int_0^T \big\Vert \< x \>^{\mu} G \big\Vert^2 d t .
\end{equation}

$ii)$ Let
\begin{equation} \label{W18}
\left\{ \begin{aligned}
&( \partial_t^2 + P ) u = G ,   \\
&( u_{\vert_{t=0}} , \partial_t u_{\vert_{t=0}} ) = 0 .
\end{aligned} \right.
\end{equation}
Then we have, for all $0 < \varepsilon < \mu$,
\begin{equation} \label{W19}
\int_0^T \big\Vert \< x \>^{- \mu} u' \big\Vert^2 d t \lesssim ( F^{\varepsilon}_{\mu} (T) )^{2} \int_0^T \big\Vert \< x \>^{\mu} G \big\Vert^2 d t .
\end{equation}
\end{proposition}

\begin{proof}
$i)$ We have
\begin{align*}
\int_0^T \big\Vert \< x \>^{- \mu} v \big\Vert^2 d t \lesssim& \int_0^T \big\Vert \< x \>^{- \mu} \Psi (P \leq 1/C) v \big\Vert^2 d t + \int_0^T \big\Vert \< x \>^{- \mu} \Psi (1/C \leq P \leq C) v \big\Vert^2 d t  \\
&+ \int_0^T \big\Vert \< x \>^{- \mu} \Psi (P \geq C) v \big\Vert^2 d t =: I_1 + I_2 + I_{3} .
\end{align*}

$\bullet$ We first estimate $I_1$. Let $\varphi$, $\widetilde{\varphi}$ be as in \eqref{c13}. By Proposition \ref{PM1}, we know that
\begin{equation*}
\big\Vert \< x \>^{- \mu} \varphi ( \lambda P) v \big\Vert^2 = \big\Vert \< x \>^{- \mu} \widetilde{\varphi} ( \lambda P) \< {\mathcal A}_{\lambda} \>^{\mu} \< {\mathcal A}_{\lambda} \>^{- \mu} \varphi ( \lambda P) v \big\Vert^2 \lesssim \lambda^{- \mu + \varepsilon_1} \big\Vert \< {\mathcal A}_{\lambda} \>^{- \mu} \varphi ( \lambda P) v \big\Vert^2 .
\end{equation*}
Therefore, we have
\begin{align*}
I_1 \lesssim & \sum_{\lambda \text{ dyadic}} \lambda^{- \mu + \varepsilon_1} \int_0^T \big\Vert\< {\mathcal A}_{\lambda} \>^{- \mu} \varphi ( \lambda P) v (t) \big\Vert^2 d t   \\
= & \sum_{\lambda \text{ dyadic}} \lambda^{- \mu + \varepsilon_1 + 1/2} \int_0^{\lambda^{-1/2} T} \big\Vert \< {\mathcal A}_{\lambda} \>^{- \mu} \varphi ( \lambda P) v ( \lambda^{1/2} s ) \big\Vert^2 d s .
\end{align*}
Now observe that $\widetilde{v} (s) = v( \lambda^{1/2} s)$ is solution of the equation
\begin{equation*}
\left\{ \begin{aligned}
&(i \partial_s - ( \lambda P )^{1/2}) \widetilde{v} =& \lambda^{1/2} G ( \lambda^{1/2} s ) ,  \\
&\widetilde{v}_{\vert_{s=0}} = 0 .
\end{aligned} \right.
\end{equation*}
We now apply Corollary \ref{c9} with $H = ( \lambda P )^{1/2}$. Using also again Proposition \ref{PM1}, we obtain
\begin{align*}
\int_0^{\lambda^{-1/2} T} \big\Vert \< {\mathcal A}_{\lambda} \>^{- \mu} & \varphi ( \lambda P) v ( \lambda^{1/2} s) \big\Vert^2 d s   \\
\lesssim & ( F_{\mu}^{\varepsilon_2} ( \lambda^{-1/2} T ) )^{2} \lambda \int_0^{\lambda^{-1/2} T} \big\Vert \< {\mathcal A}_{\lambda} \>^{\mu} \varphi ( \lambda P) G ( \lambda^{1/2} s) \big\Vert^2 d s   \\
\lesssim & ( F_{\mu}^{\varepsilon_2} ( \lambda^{-1/2} T ) )^{2} \lambda^{1 - \mu + \varepsilon_3} \int_0^{\lambda^{-1/2} T} \big\Vert \< x \>^{\mu} G ( \lambda^{1/2} s ) \big\Vert^2 d s   \\
= & ( F_{\mu}^{\varepsilon_2} ( \lambda^{-1/2} T ) )^{2} \lambda^{1/2 - \mu + \varepsilon_3} \int_0^{T} \big\Vert \< x \>^{\mu} G (t) \big\Vert^2 d t .
\end{align*}
Thus,
\begin{equation*}
I_1 \lesssim \sum_{\lambda \text{ dyadic}} \lambda^{1 - 2 \mu + \varepsilon_{1} + \varepsilon_3} ( F_{\mu}^{\varepsilon_2} ( \lambda^{-1/2} T ) )^{2} \int_0^{T} \big\Vert \< x \>^{\mu} G (t) \big\Vert^2 d t .
\end{equation*}
If $\mu\leq 1/2$, then we see that 
\begin{equation*}
I_1 \lesssim \sum_{\lambda \text{ dyadic}} \lambda^{\varepsilon_1 + \varepsilon_3 - 2 \varepsilon_2} T^{2 ( 1-2 \mu + 2 \varepsilon_2 )} \int_0^{T} \big\Vert \< x \>^{\mu} G (t) \big\Vert^2 d t .
\end{equation*}
Once $0 < \varepsilon_2 < \mu$ fixed, it is therefore sufficient to choose $\varepsilon_1 , \varepsilon_3$ small enough such that $\varepsilon_1 + \varepsilon_3 < 2 \varepsilon_2$. If $\mu > 1/2$, we choose $\varepsilon_1 , \varepsilon_3$ small enough such that $\varepsilon_1 + \varepsilon_3 < 2\mu - 1$. Then,
\begin{equation*}
I_1 \lesssim \int_0^{T} \big\Vert \< x \>^{\mu} G (t) \big\Vert^2 d t .
\end{equation*}

$\bullet$ We now study $I_{2}$. Part $iii)$ of Proposition \ref{PM2} implies
\begin{equation*}
I_{2} \lesssim \int_0^T \big\Vert \< {\mathcal A} \>^{- \mu} \Psi (1/C \leq P \leq C) v \big\Vert^2 d t .
\end{equation*}
As in the proof of Proposition \ref{propEW2}, Proposition \ref{PM2} gives us a finite number of open intervals $I_{k}$, $k =1 , \ldots , K$, satisfying \eqref{c12} and
\begin{equation*}
[ 1/C , C] \subset \bigcup_{k=1}^{K} I_{k} .
\end{equation*}
Then, applying Corollary \ref{c9} on each $I_{k}$ (slightly reduced) and using Proposition \ref{PM2}, we obtain
\begin{equation*}
I_{2} \lesssim ( F^{\varepsilon}_{\mu} (T) )^{2} \int_0^T \big\Vert \< x \>^{\mu} G \big\Vert^2 d t .
\end{equation*}

$\bullet$ We finally estimate $I_{3}$. Proposition \ref{a26} and Corollary \ref{c9} yield
\begin{align*}
I_{3} & \lesssim \int_0^T \big\Vert \< A_{\infty} \>^{- \mu} \Psi (P \geq C) v \big\Vert^2 d t \lesssim ( F^{\varepsilon}_{\mu} (T) )^{2}\int_0^T \big\Vert \< A_{\infty} \>^{\mu} G \big\Vert^2 d t   \\
& \lesssim ( F^{\varepsilon}_{\mu} (T) )^{2} \int_0^T \big\Vert \< x \>^{\mu} G \big\Vert^2 d t .
\end{align*}

$ii)$ We first write \eqref{W17} as a first order system
\begin{align*}
i \partial_t \left( \begin{array}{c} u \\ \partial_t u \end{array} \right) = R \left( \begin{array}{c} u \\ \partial_t u \end{array} \right) + i \left( \begin{array}{c} 0 \\ G \end{array} \right) , \quad \left( \begin{array}{c} u \\ \partial_t u \end{array} \right) \vert_{t=0} = 0 .
\end{align*}
It is sufficient to estimate, for $\widetilde{\mu} < \mu$,
\begin{equation*}
\int_0^T \left\Vert \< x \>^{-\widetilde{\mu}} M \left( \begin{array}{c} u \\ \partial_t u \end{array} \right) \right\Vert^2_{L^2 \times L^2} d t = \int_0^T \left\Vert \< x \>^{-\widetilde{\mu}} M U^* U \left( \begin{array}{c} u \\ \partial_t u \end{array} \right) \right\Vert^2_{L^2 \times L^2} d t ,
\end{equation*}
with $M$ defined in \eqref{c14}. But $v = U \left( \begin{array}{c} u \\ \partial_t u \end{array} \right)$ solves
\begin{equation*}
( i \partial_t - L ) v = i U \left( \begin{array}{c} 0 \\ G \end{array} \right) , \quad v_{\vert_{t=0}} = 0 .
\end{equation*}
By \eqref{c14} and part $i)$ of the proposition, we find
\begin{align*}
\int_0^T \left\Vert \< x \>^{-\widetilde{\mu}} M \left( \begin{array}{c} u \\ \partial_t u \end{array} \right) \right\Vert^2_{L^2 \times L^2} d t \lesssim & \int_0^T \big\Vert \< x \>^{-\widetilde{\mu}} v \big\Vert^2_{L^2 \times L^2} d t  \\
\lesssim & ( F^{\varepsilon}_{\widetilde{\mu}} (T) )^{2}  \int_0^T \left\Vert \< x \>^{- \widetilde{\mu}} U \left( \begin{array}{c} 0 \\ G \end{array} \right) \right\Vert^2_{L^2 \times L^2} d t  \\
= & ( F^{\varepsilon}_{\widetilde{\mu}} (T) )^{2}  \int_0^T \big\Vert\< x \>^{-\widetilde{\mu}} G \big\Vert^2_{L^2} d t .
\end{align*}
which gives $ii)$ thanks to Remark \ref{RW2}.
\end{proof}

We now want to prove Theorem \ref{TW1} for general $N$. In contrast to the Minkowski case, this does not follow directly from the case $N=0$ because the vector fields $\widetilde{\Omega}$, $\widetilde{\partial}_x$ do not commute with the equation. We will therefore need the form of certain commutators. As in \eqref{c17}, a term $r_{j}$ or $\widetilde{r}_{j}$, $j\in \N$, will denote a smooth function such that
\begin{align*}
\partial^{\alpha}_{x} r_{j} (x) &= \CO \big( \< x \>^{-\rho - j - \vert \alpha \vert} \big) ,   \\
\partial^{\alpha}_{x} \widetilde{r}_{j} (x) &= \CO \big( \< x \>^{- j - \vert \alpha \vert} \big) .
\end{align*}
These functions can change from line to line. Direct computations give

\begin{lemma}\sl \label{LW2}
We have
\begin{align*}
&\big[ \widetilde{\partial}_{j} , \widetilde{\partial}_{k} \big] = r_{1} \widetilde{\partial} ,
&& \big[ \widetilde{\Omega}^{j, k} ,P \big] = r_{0} \widetilde{\partial} \widetilde{\partial}+r_{1} \widetilde{\partial}  \\
&\big[ \widetilde{\partial}^{*}_{j} , \widetilde{\partial}_{k} \big] = \widetilde{\partial}^{*} r_{1} + r_{2} ,
&&\big[ \widetilde{\Omega}^{j, k} , \widetilde{\partial}_{\ell} \big] = \widetilde{r}_{0} \widetilde{\partial} ,  \\
&\big[ P , \widetilde{\partial}_{\ell} \big] = \widetilde{\partial}^{*} r_{1} \widetilde{\partial} + r_{2} \widetilde{\partial} .
\end{align*}
As before, we have not written the sum over the indexes on the right hand sides.
\end{lemma}

We now observe that the vector fields $\widetilde{\partial}_j$ can be replaced by powers of $P$.

\begin{lemma}\sl  \label{LW3}
For $0 < \mu \leq 3/2$ and $n \geq 2$, we have
\begin{equation} \label{W7}
\big\Vert \< x \>^{- \mu} \widetilde{\partial}_{j_1} \cdots \widetilde{\partial}_{j_n} u \big\Vert \lesssim \sum_{j=0}^{\lfloor \frac{n-1}{2} \rfloor} \sum_{q =1}^{d} \big\Vert \< x \>^{- \mu} \widetilde{\partial}_{q} P^j u \big\Vert + \sum_{j=1}^{\lfloor \frac{n}{2} \rfloor} \big\Vert \< x \>^{- \mu} P^j u \big\Vert .
\end{equation}
\end{lemma}

\begin{proof}
We first show
\begin{equation} \label{W8}
\big\Vert \< x \>^{- \mu} \widetilde{\partial}_k \widetilde{\partial}_{\ell} u \big\Vert \lesssim \Vert \< x \>^{- \mu} P u \Vert + \sum_{q = 1}^{d} \big\Vert \< x \>^{- \mu} \widetilde{\partial}_{q} u \big\Vert .
\end{equation}
Indeed, we have
\begin{equation*}
\big\Vert \< x \>^{- \mu} \widetilde{\partial}_k \widetilde{\partial}_{\ell} u \big\Vert \lesssim \big\Vert \< x \>^{- \mu} \widetilde{\partial}_k (P+1)^{-1} \widetilde{\partial}_{\ell} (P+1) u \big\Vert + \big\Vert \< x \>^{- \mu} \widetilde{\partial}_k (P+1)^{-1} \big[ P , \widetilde{\partial}_{\ell} \big] u \big\Vert =: A + B .
\end{equation*}
We estimate $A$.
\begin{equation*}
A \leq \big\Vert \< x \>^{- \mu} \widetilde{\partial}_k (P+1)^{-1} \widetilde{\partial}_{\ell} P u \big\Vert + \big\Vert \< x \>^{- \mu} \widetilde{\partial}_k (P+1)^{-1} \widetilde{\partial}_{\ell} u \big\Vert .
\end{equation*}
Noting that $\< x \>^{- \mu} \widetilde{\partial}_k (P+1)^{-1} \widetilde{\partial}_{\ell} \< x \>^{\mu}$ and $\< x \>^{- \mu} \widetilde{\partial}_k (P+1)^{-1} \< x \>^{\mu}$ are bounded by Proposition~\ref{b45} and Lemma~\ref{b20}, we obtain
\begin{equation} \label{W9}
A \lesssim \Vert \< x \>^{- \mu} P u \Vert + \big\Vert \< x \>^{- \mu} \widetilde{\partial}_{\ell} u \big\Vert .
\end{equation}
Now, recall from Lemma \ref{LW2} that
\begin{equation*}
\big[ P , \widetilde{\partial}_{\ell} \big] = \widetilde{\partial}^{*} r_{1} \widetilde{\partial} + r_{2} \widetilde{\partial} .
\end{equation*}
Thus, as for \eqref{W9}, we see that
\begin{equation} \label{W10}
B \lesssim \sum_j \big\Vert \< x \>^{- \mu} \widetilde{\partial}_j u \big\Vert .
\end{equation}
The inequalities \eqref{W9}, \eqref{W10} give \eqref{W8}. We will show \eqref{W7} by induction over $n\ge2$. For $n=2$ this is exactly \eqref{W8}. Assume $n \geq 3$. Using \eqref{W8}, we obtain
\begin{equation*}
\big\Vert \< x \>^{- \mu} \widetilde{\partial}_{j_1} \widetilde{\partial}_{j_2} \cdots \widetilde{\partial}_{j_n} u \big\Vert \lesssim \big\Vert \< x \>^{- \mu} P \widetilde{\partial}_{j_3} \cdots \widetilde{\partial}_{j_n} u \big\Vert + \sum_{k=1}^{d} \big\Vert \< x \>^{- \mu} \widetilde{\partial}_k \widetilde{\partial}_{j_3} \cdots \widetilde{\partial}_{j_n} u \big\Vert .
\end{equation*}
For the second term, we can use the induction hypothesis. For the first term we commute $P$ through $\widetilde{\partial}_{j_3} \cdots \widetilde{\partial}_{j_n}$. The commutators give terms which can be estimated by terms of the form $\Vert \< x \>^{- \mu} \widetilde{\partial}_{k_m} \cdots \widetilde{\partial}_{k_n} u \Vert$, with $2 \leq m \leq n$, which themselves can be estimated by the induction hypothesis. It remains to consider the term $\Vert\<x\>^{-\mu}\widetilde{\partial}_{j_3} \cdots \widetilde{\partial}_{j_n}Pu\Vert$,
which can either be kept ($n=3$) or be estimated applying the induction hypothesis to $P u$ rather than to $u$.
\end{proof}

In order to show \eqref{W1}, it is sufficient to use vector fields in $X$. This is shown in the next lemma.

\begin{lemma}\sl \label{LW4}
Assume $\rho > 1$. Let $1/2 \leq \mu \leq 1$, $j \in \frac{1}{2} \N$, $\beta$ be a multi-index and $N = 2 j + \vert \beta \vert$. Then, for all $\varepsilon>0$, there exists $\eta_{\varepsilon}> 1/2$ such that
\begin{align}
\< F_{\mu}^{\varepsilon}(T) & \>^{-1} \big\Vert \< x \>^{- \mu} (P^j \widetilde{\Omega}^{\beta}u)' \big\Vert_{L^2( [0 , T] \times \R^{d} )} \nonumber \\
\lesssim & \sum_{\vert\alpha\vert\leq N} \bigg( \big\Vert (Y^{\alpha} u) '(0, \cdot ) \big\Vert_{L^2 ( \R^{d} )} + \int_0^T \Vert Y^{\alpha} G \Vert_{L^2 ( \R^{d} )} d t + \big\Vert \< x \>^{-\eta_{\varepsilon}} (X^{\alpha}u)' \big\Vert_{L^2( [0 , T] \times \R^{d} )} \bigg) . \label{W11}
\end{align}
Moreover, for $\rho = 1$ and $\varepsilon >0$, the same inequality holds with $\< F_{\mu}^{\varepsilon}(T) \>^{-1}$ replaced by $\< T \>^{- \varepsilon}$.
\end{lemma}

\begin{proof}
The inequality will be proven by induction over $\vert \beta \vert$. Assume first $\rho >1$. Since the wave equation commutes with $P$, the case $\vert \beta \vert =0$ follows from Theorem \ref{TSLW2} and Lemma \ref{b53}. Assume now $\vert \beta \vert \geq 1$ and let $v= P^j \widetilde{\Omega}^{\beta} u$. The function $v$ fulfills the following equation
\begin{equation}
\left\{ \begin{aligned}
&( \partial_t^2 + P ) v = P^j \widetilde{\Omega}^{\beta} G + P^j \big[ P , \widetilde{\Omega}^{\beta} \big] u ,   \\
&( v_{\vert_{t=0}} , \partial_t v_{\vert_{t=0}} ) = ( P^j \widetilde{\Omega}^{\beta} u_0 , P^j \widetilde{\Omega}^{\beta} u_1 ) .
\end{aligned} \right.
\end{equation}
Let $v_1 , v_2$ be the solutions of
\begin{equation} \label{W13}
\left\{ \begin{aligned}
&( \partial_t^2 + P ) v_1 = P^j \widetilde{\Omega}^{\beta} G ,  \\
& ( v_{1} {}_{\vert_{t=0}} , \partial_t v_{1} {}_{\vert_{t=0}} ) = ( P^j \widetilde{\Omega}^{\beta} u_0 , P^j \widetilde{\Omega}^{\beta} u_1 ) ,
\end{aligned} \right.
\end{equation}
\begin{equation} \label{W14}
\left\{ \begin{aligned}
&( \partial_t^2 + P ) v_2 = P^j [ P , \widetilde{\Omega}^{\beta} ] u , \\
&( v_{2} {}_{\vert_{t=0}} , \partial_t v_{2} {}_{\vert_{t=0}} ) = 0 .
\end{aligned} \right.
\end{equation}
Clearly $v = v_1 + v_2$. We have, for all $\widetilde{\mu} < \mu$,
\begin{equation*}
\big\Vert \< x \>^{- \mu} v_1 ' \big\Vert_{L^2( [0 , T] \times \R^{d} )} \lesssim ( F_{\widetilde{\mu}}^{\varepsilon} (T) )^{1/2} \bigg( \big\Vert ( P^j \widetilde{\Omega}^{\beta} u)' (0, \cdot ) \big\Vert_{L^2 ( \R^{d} )} + \int_0^T \big\Vert P^j \widetilde{\Omega}^{\beta} G \big\Vert_{L^2 ( \R^{d} )} d t \bigg) ,
\end{equation*}
where we have used Theorem \ref{TSLW2}. If $\mu > 1/2$, we choose $\widetilde{\mu} > 1/2$. We further estimate, by Proposition \ref{PW4},
\begin{equation*}
( F_{\widetilde{\mu}}^{\varepsilon} (T) )^{-1} \Vert \< x \>^{- \mu} v_2 ' \Vert_{L^2( [0 , T] \times \R^{d} )} \lesssim \big\Vert \< x \>^{\widetilde{\mu}} P^j [ P, \widetilde{\Omega}^{\beta} ] u \big\Vert_{L^2( [0 , T] \times \R^{d} )} .
\end{equation*}
Using Lemma \ref{LW2}, we see that $\< x \>^{\widetilde{\mu}} P^j [ P , \widetilde{\Omega}^{\beta} ] u$ is a sum of terms of the form
\begin{equation*}
\< x \>^{\widetilde{\mu} - \rho} \widetilde{\partial}_{k_1} \cdots \widetilde{\partial}_{k_{m}} \widetilde{\Omega}^{\gamma} u ,
\end{equation*}
with $1 \leq m \leq 2j+2$ and $\vert \gamma \vert \leq \vert \beta \vert - 1$. Using Lemma \ref{LW3}, we see that these terms can be estimated in norm by terms of the form
\begin{equation*}
\big\Vert \< x \>^{\widetilde{\mu} - \rho} \widetilde{\partial}_\ell (P^{q} \widetilde{\Omega}^{\gamma} u ) \big\Vert \quad \text{or} \quad \big\Vert \< x \>^{\widetilde{\mu} - \rho} P^{r} \widetilde{\Omega}^{\gamma} u \big\Vert ,
\end{equation*}
with $q,r \in \N$, $0 \leq q \leq (m-1)/2$ and $1 \leq r \leq m/2$. Applying Lemma \ref{LW5}, we see that we can replace $P^{1/2}$ in the second term by partial derivatives and apply the induction hypothesis with $\rho - \widetilde{\mu} > 1/2$.

In the case $\rho =1$, it is enough to choose $\widetilde{\mu} = 1/2 - \delta$ with $\delta >0$ small.
\end{proof}

\begin{proof}[Proof of Theorem \ref{TW1}]
The energy term is easily estimated by the observation that $\partial_t$ and $P$ commute with the equation. The same way, note that we can restrict our attention to vector fields in $Y$ for the second term. Also, by Lemma \ref{LW2}, we can arrange for that the vector fields $\widetilde{\partial}_{x}$ are always on the left of the vector fields $\widetilde{\Omega}$. Using Lemma \ref{LW2}, we see that we can replace $Y^{\alpha} u '$ by $(Y^{\alpha}u)'$. Using Lemma \ref{LW5}, Lemma \ref{LW3} and Lemma \ref{LW4}, we see that it is sufficient to estimate
\begin{equation*}
\< F_{\mu}^{\varepsilon} (T) \>^{-1} \big\Vert \< x \>^{- \mu} P^j v \big\Vert_{L^2( [0 , T] \times \R^{d} )} ,
\end{equation*}
in the case $\rho > 1$ and
\begin{equation*}
\< T \>^{- \varepsilon} \big\Vert \< x \>^{- \mu} P^j v \big\Vert_{L^2( [0 , T] \times \R^{d} )} ,
\end{equation*}
in the case $\rho = 1$. These terms can be estimated by Theorem \ref{TSLW2}, because $P$ commutes with the equation.
\end{proof}

\section{Proof of the nonlinear result} \label{sec6}

In this section we will prove the main theorem, Theorem \ref{TSLW}. The proof of the result will follow closely the arguments of Keel, Smith and Sogge in the Minkowski case (see \cite{KeSmSo02_01}). We start with the now standard Sobolev estimate (see \cite{Kl85_01}).

\begin{lemma}\sl   \label{LN1}
Suppose that $h \in C^{\infty} ( \R^d )$. Then, for $R>1$,
\begin{equation} \label{N1}
\Vert h \Vert_{L^{\infty} (R/2 \leq \vert x \vert \leq R)} \lesssim R^{\frac{1-d}{2}} \sum_{\vert \alpha \vert \leq \left \lceil \frac{d-1}{2} \right\rceil +1} \Vert Y^{\alpha} h \Vert_{L^2(R/4 \leq \vert x \vert \leq 2R)}.
\end{equation}
\end{lemma}

We now define the bilinear form $\widetilde{Q}$ by $\widetilde{Q}(u',u')=Q(u')$. The following estimate for the nonlinear part will be crucial.

\begin{lemma}\sl  \label{LN2}
Let $\mu_d = \frac{d-1}{4}$. Then, for $L \geq \max \big( 2 \left( \left\lceil \frac{d-1}{2} \right\rceil + 1 \right) , \vert \beta \vert \big)$, we have
\begin{equation*}
\big\Vert Z^{\beta} \widetilde{Q} (u',v') \big\Vert_{L^2 ( \R^{d} )}^{2} \lesssim \bigg( \sum_{\vert \alpha \vert \leq L} \big\Vert \< x \>^{- \mu_d} Z^{\alpha} u' \big\Vert^2_{L^2 ( \R^{d} )} \bigg) \bigg( \sum_{\vert \alpha \vert \leq L} \big\Vert \< x \>^{- \mu_d} Z^{\alpha} v' \big\Vert^2_{L^2 ( \R^{d} )} \bigg) .
\end{equation*}
\end{lemma}

\begin{proof}
We clearly have the pointwise bound:
\begin{align*}
\big\vert Z^{\beta} \widetilde{Q} (u',v') (s,x) \big\vert \lesssim & \bigg( \sum_{\vert\alpha\vert\leq L} \big\vert Z^{\alpha}u'(s,x) \big\vert \bigg) \bigg( \sum_{\vert \alpha \vert \leq \left \lfloor \frac{L}{2} \right \rfloor} \big\vert Z^{\alpha} v' (s,x) \big\vert \bigg)   \\
&+ \bigg( \sum_{\vert\alpha\vert\leq L} \big\vert Z^{\alpha} v'(s,x) \big\vert \bigg) \bigg( \sum_{\vert \alpha \vert \leq \left \lfloor \frac{L}{2} \right \rfloor} \big\vert Z^{\alpha} u' (s,x) \big\vert \bigg) .
\end{align*}
We only estimate the first term. Using Lemma \ref{LN1} for a given $R=2^j$, $j\geq 0$, we get
\begin{align*}
\big\Vert Z^{\beta} & \widetilde{Q} (u',v') \big\Vert^{2}_{L^2 ( \{ \vert x \vert \in [ 2^j , 2^{j+1} [ \} )}  \\
&\lesssim 2^{j (1-d)} \sum_{\vert \alpha \vert \leq L} \big\Vert Z^{\alpha} u' \big\Vert^{2}_{L^2 ( \{ \vert x \vert \in [ 2^j , 2^{j+1} [ \} )} \sum_{\vert \alpha \vert \leq \left\lfloor \frac{L}{2} \right\rfloor + \left\lceil \frac{d-1}{2} \right\rceil +1} \big\Vert Z^{\alpha} v' \big\Vert^{2}_{L^2 ( \{ \vert x \vert \in [ 2^{j-1} , 2^{j+2} [ \} )}    \\
&\lesssim \sum_{\vert \alpha \vert \leq L} \big\Vert \< x \>^{- \mu_d} Z^{\alpha} u' \big\Vert^{2}_{L^2 ( \{ \vert x \vert \in [ 2^{j} , 2^{j+1} [ \} )} \sum_{\vert \alpha \vert \leq L} \big\Vert \< x \>^{- \mu_d} Z^{\alpha} v' \big\Vert^{2}_{L^2 ( \{ \vert x \vert \in [ 2^{j-1} , 2^{j+2} [ \} )}   \\
&\lesssim \sum_{\vert \alpha \vert \leq L} \big\Vert \< x \>^{- \mu_d} Z^{\alpha} u' \big\Vert^{2}_{L^2 ( \{ \vert x \vert \in [ 2^{j} , 2^{j+1} [ \} )}  \sum_{\vert \alpha \vert \leq L} \big\Vert \< x \>^{- \mu_d} Z^{\alpha} v' \big\Vert^{2}_{L^2 ( \R^{d} )} .
\end{align*}
We also have the bound
\begin{eqnarray*}
\big\Vert Z^{\beta} \widetilde{Q} (u',v') \big\Vert_{L^2 ( \{ \vert x \vert < 1 \} )}^{2} \lesssim \sum_{\vert \alpha \vert \leq L} \big\Vert Z^{\alpha} u' \big\Vert^2_{L^2 ( \{ \vert x \vert <2 \} )} \sum_{\vert \alpha \vert \leq L} \big\Vert Z^{\alpha} v' \big\Vert^2_{L^2 ( \{ \vert x \vert <2 \} )} .
\end{eqnarray*}
Summing over $j$ gives the lemma.
\end{proof}

\begin{proof}[Proof of Theorem \ref{TSLW}]
We follow \cite{KeSmSo02_01}. Let $u_{-1}=0$. We define $u_k$, $k \in \N$ inductively by letting $u_k$ solve
\begin{equation}   \label{N3}
\left\{ \begin{aligned}
&\Box_{\mathfrak{g}} u_k = Q ( u_{k-1} ' ) ,   \\
&( u_k {}_{\vert_{t=0}} , \partial_t u_k {}_{\vert_{t=0}} ) = ( u_0 , u_1 ) .
\end{aligned} \right.
\end{equation}
For $T > 0$, we denote
\begin{equation*}
M_k (T) = \sup_{0 \leq t \leq T} \sum_{1 \leq i + j \leq M+1} \big\Vert \partial_{t}^{i} P^{j/2} u_k \big\Vert_{L^{2} ( \R^{d} )} + \sum_{\vert \alpha \vert \leq M} K_{n} (T)^{-1} \big\Vert \< x \>^{- \mu_d} Z^{\alpha} u_k' \big\Vert_{L^2( [0 , T] \times \R^{d} )} ,
\end{equation*}
with
\begin{equation*}
K_{n} (T) =
\left\{ \begin{aligned}
&T^{1/n} &&d=3 \text{ or } \rho =1 ,  \\
&1 &&d\geq 4 \text{ and } \rho > 1 . 
\end{aligned} \right.
\end{equation*}
Using Theorem \ref{TW1}, we see that there exists a constant $C_0$ such that 
\begin{equation*}
M_0 (T) \leq C_0 \delta ,
\end{equation*}
for any $T$. We claim that, for $k \geq 1$, we have
\begin{equation} \label{c10}
M_k ( T_{\delta} ) \leq 2 C_0 \delta ,
\end{equation}
for $\delta$ sufficiently small and $T_{\delta}$ appropriately chosen. We will prove this inductively. Assume that the bound holds for $k-1$. By Theorem \ref{TW1}, we have, for $\delta$ small enough,
\begin{align*}
M_k ( T_{\delta} ) &\leq C_0 \delta + C \sum_{\vert \alpha \vert \leq M} \int_0^{T_{\delta}} \big\Vert Z^{\alpha} Q ( u_{k-1} ') (s, \cdot ) \big\Vert_{L^2 ( \R^d )} d s   \\
& \leq C_0 \delta + C \sum_{\vert \alpha \vert \leq M} \int_0^{T_{\delta}} \big\Vert \< x \>^{- \mu_d} Z^{\alpha} u_{k-1}' \big\Vert^2_{L^2 ( \R^d )} d s   \\
& \leq C_0 \delta + C K_{n} ( T_{\delta} ) M_{k-1}^2 ( T_{\delta} )   \\
& \leq C_0 \delta + C K_{n} ( T_{\delta} ) ( 2 C_0 \delta)^2 ,
\end{align*}
where we have also used Lemma \ref{LN2} and the induction hypothesis. Note that, to estimate the term $\Vert ( Z^{\alpha} u_k ) ' ( 0, \cdot ) \Vert_{L^2}$, we might have to use the equation and Lemma \ref{LN2}. We therefore need $\delta$ to be small enough. Then, to prove \eqref{c10}, it is enough to have
\begin{equation} \label{N4}
C_0 \delta + C K_{n} ( T_{\delta} ) ( 2 C_0 \delta )^2 \leq 2 C_0 \delta \Longleftrightarrow 4 C C_0 K_{n} ( T_{\delta} ) \delta \leq 1 .
\end{equation}
Therefore, we find:
\begin{itemize}
\item If $d=3$ or $\rho =1$, the estimate holds with $T_{\delta} = c_{n} \delta^{- n}$ and $c_{n}$ small enough.
\item If $d\geq 4$ and $\rho >1$, \eqref{N4} is fulfilled if $\delta$ is sufficiently small and we can take $T_{\delta}=\infty$.
\end{itemize}
To show that the sequence $u_k$ converges, we estimate the quantity
\begin{align*}
A_k (T) =  \sup_{0 \leq t \leq T} \sum_{1 \leq i + j \leq M+1} \big\Vert & \partial_{t}^{i} P^{j/2} ( u_k - u_{k-1} ) \big\Vert_{L^2 ( \R^d )}  \\
& + \sum_{\vert \alpha \vert \leq M} K_{n} (T)^{-1} \big\Vert \< x \>^{- \mu_d} Z^{\alpha} ( u'_k - u'_{k-1} ) \big\Vert_{L^2( [0,T] \times \R^{d} )}.
\end{align*}
It is clearly sufficient to show
\begin{equation} \label{N5}
A_k (T) \leq \frac{1}{2} A_{k-1} (T) .
\end{equation}
Using Lemma \ref{LN2} and repeating the above arguments, we obtain
\begin{align*}
A_k ( T_{\delta} ) \leq \widetilde{C} \sum_{\vert \alpha \vert \leq M} \int_0^{T_{\delta}} & \big\Vert \< x \>^{- \mu_d} Z^{\alpha} ( u'_{k-1} - u'_{k-2} ) \big\Vert_{L^2 ( \R^d )}  \\
& \times \sum_{\vert \alpha \vert \leq M} \Big( \big\Vert \< x \>^{- \mu_d} Z^{\alpha} u'_{k-1} \big\Vert_{L^2 ( \R^d )} + \big\Vert \< x \>^{- \mu_d} Z^{\alpha} u'_{k-2} \Vert_{L^2 ( \R^d )} \Big) d s .
\end{align*}
By the Cauchy--Schwarz inequality, we conclude that 
\begin{equation*}
A_k ( T_{\delta} ) \leq \widetilde{C} K_{n} ( T_{\delta} ) ( M_{k-1} ( T_{\delta} ) + M_{k-2} ( T_{\delta} ) ) A_{k-1} (T_{\delta} ) .
\end{equation*}
Using \eqref{c10}, the above inequality leads to \eqref{N5} if $\delta$ is small enough. Uniqueness and $C^2$ property of the solution follow from \cite[Theorem~6.4.10, Theorem~6.4.11]{Ho97_01} using that the constructed solution is in $H^{M+1}_{\text{loc}} ( \R^{d+1} ) \subset C^2 ( \R^{d+1} )$. Note also that the solution is bounded in $C^2$ on the interval $[0,T_{\delta} ]$.
\end{proof}

\appendix

\section{Regularity} \label{a29}

Here, we give some results concerning the regularity with respect to an operator. More details can be found in the book of Amrein, A. Boutet de Monvel and Georgescu \cite{AmBoGe96_01} and in the paper of C. G\'erard and Georgescu \cite{GeGe99_01}. We start with a useful characterization of the regularity $C^{1} (A)$.

\begin{theorem}[{\cite[Theorem~6.2.10]{AmBoGe96_01}}]\sl  \label{a15}
Let $A$ and $H$ be self-adjoint operators on a Hilbert space ${\mathcal H}$. Then $H$ is of class $C^{1}(A)$ iff the following conditions are satisfied: \parskip=0in
\begin{enumerate}[$i)$]
\item there is a constant $c< \infty$ such that for all $u \in D(A)
\cap D(H)$, 
\begin{equation*}
\vert(Au,Hu) - (Hu,Au) \vert \leq c \left( \Vert H u \Vert^{2} + \Vert u\Vert^{2} \right), 
\end{equation*}

\item for some $z \in \C\backslash \sigma (H)$, the set $\{ u \in
D(A); \ (H-z)^{-1} u \in D(A) \text{ and } (H-\bar{z})^{-1} u \in D(A)
\}$ is a core for $A$. 
\end{enumerate}
If $H$ is of class $C^{1} (A)$, then the following is true: \parskip=0in
\begin{enumerate}[$i)$]
\item The space $(H-z)^{-1} D(A)$ is independent of $z\in \C \backslash
\sigma (H)$ and contained in $D(A)$. It is a core for $H$ and a dense
subspace of $D(A) \cap D(H)$ for the intersection topology (i.e. the
topology associated to the norm $ \Vert H u\Vert + \Vert Au\Vert + \Vert u \Vert)$.

\item The space $D(A) \cap D(H)$ is a core for $H$ and the form
$[A,H]$ has a unique extension to a continuous sesquilinear form on
$D(H)$ (equipped with the graph topology). If this extension is
denoted by $[A,H]$, the following identity holds on ${\mathcal H}$ (in
the form sense): 
\begin{equation*}
\left[ A, (H-z)^{-1} \right] = - (H-z)^{-1} [A,H] (H-z)^{-1},
\end{equation*}
for $z \in \C \backslash \sigma (H)$.
\end{enumerate}
\end{theorem}

We also have the following theorem coming from \cite[Theorem~6.3.4]{AmBoGe96_01}.

\begin{theorem}\sl   \label{a16}
Let $A$ and $H$ be self-adjoint operators in a Hilbert space ${\mathcal H}$. Assume that the unitary one-parameter group $\{ \exp (i A \tau) \}_{\tau \in \R}$ leaves the domain $D(H)$ of $H$ invariant. Then $H$ is of class $C^{1}(A)$ iff $[H,A]$ is bounded from $D(H)$ to $D(H)^{*}$.
\end{theorem}

A criterion for the above assumption to be satisfied is given by the following result of Georgescu and C. G\'erard.

\begin{lemma}[{\cite[Lemma~2]{GeGe99_01}}]\sl  \label{a17}
Let $A$ and $H$ be self-adjoint operators in a Hilbert space ${\mathcal H}$. Let $H \in C^{1} (A)$ and suppose that the commutator $[i H , A]$ can be extended to a bounded operator from $D(H)$ to ${\mathcal H}$. Then $e^{i t A}$ preserves $D(H)$. 
\end{lemma}

In this paper, we will use the following characterization of the regularity $C^{2} (A)$.

\begin{remark}\sl \label{a7}
From Section 6.2 of \cite{AmBoGe96_01}, it is known that $H$ if of class $C^{2} (A)$ if the following conditions hold: %\parskip=0in

$i)$ For some $z \in \C \setminus \sigma (H)$, the set $\{ u \in D (A) ; \ (H-z)^{-1} u \in D (A) \text{ and } (H- \overline{z} )^{-1} u \in D (A) \}$ is a core for $A$.

$ii)$ $[H,A]$ and $[[H,A],A]$ extend as bounded operators on ${\mathcal H}$.
\end{remark}

\section{Resolvent estimates at low energies}\label{b56}

\Subsection{Estimates for the free Laplacian}

We begin with some estimates for the free Laplacian $P_{0} = - \Delta$.

\begin{lemma}\sl \label{b6}
Let $\alpha >0$. Then, for all $\varepsilon >0$, we have
\begin{equation*}
\big\Vert ( \lambda P_{0} +1)^{- \alpha} u \big\Vert \lesssim \lambda^{- \min ( \alpha - \varepsilon , d/4)} \big\Vert \< x \>^{ \min ( 2 \alpha , d/2 + \varepsilon )} u \big\Vert ,
\end{equation*}
uniformly for $\lambda \geq 1$.
\end{lemma}

\begin{proof}
Here, we denote
\begin{equation*}
\Vert u \Vert_{p} = \Big( \int_{\R^{d}} \vert u (x) \vert^{p} \, d x \Big)^{1/p} ,
\end{equation*}
the standard norm on $L^{p} ( \R^{d})$. Using the H\"{o}lder inequality, we get
\begin{align*}
\big\Vert ( \lambda P_{0} +1)^{- \alpha} u \big\Vert =& \Big( \int (\lambda \xi^{2} +1)^{- 2 \alpha} \vert \widehat{u} ( \xi )\vert^{2} d \xi \Big)^{1/2}  \\
\leq& \Vert (\lambda \xi^{2} +1)^{- \alpha} \Vert_{2 p} \Vert \widehat{u} \Vert_{2 q} ,
\end{align*}
and we choose $p = \max (\frac{d}{4 \alpha} + \mu , 1)$, $\mu >0$, and $p^{-1} + q^{-1} = 1$. In particular, $2 q \geq 2$ and $4 \alpha p >d$. Then, by the Hausdorff--Young inequality, we obtain
\begin{equation} \label{b4}
\big\Vert ( \lambda P_{0} +1)^{- \alpha} u \big\Vert \lesssim \lambda^{- d/ 4p} \Vert u \Vert_{r} ,
\end{equation}
with $r^{-1} = 1 - (2q)^{-1} = 2^{-1} (1 + p^{-1} )$ satisfying $1 \leq r \leq 2$. Using one more time the H\"{o}lder inequality, we have
\begin{equation*}
\Vert u \Vert_{r} \lesssim \Big( \int \vert u \vert^{r s} \< x \>^{\beta s} d x \Big)^{1/ r s} \Big( \int \< x \>^{- \beta t} d x \Big)^{1/ r t} ,
\end{equation*}
with $s^{-1} +t^{-1} =1$ and $\beta >0$. We choose $s = 2 r^{-1}$ and $\beta = d/t + \nu$, $\nu >0$. Thus,
\begin{equation} \label{b5}
\Vert u \Vert_{r} \lesssim \big\Vert \< x \>^{\beta s /2} u \big\Vert_{2} .
\end{equation}
The coefficient $\beta s /2$ satisfies
\begin{align*}
\frac{\beta s}{2} =& \frac{d s}{2 t} + \frac{\nu s}{2} = \frac{d s}{2} - \frac{d}{2} + \frac{\nu s}{2} = \frac{d}{r} - \frac{d}{2} + \CO (\nu ) = \frac{d}{2 p} + \CO (\nu )  \\
=& \frac{d}{2} \min \Big( \Big( \frac{d}{4 \alpha} + \mu \Big)^{-1} ,1 \Big) + \CO (\nu ) = \frac{d}{2} \min \Big( \frac{4 \alpha}{d} - \frac{16 \alpha^{2} \mu}{d^{2}} + \CO ( \mu^{2} ) ,1 \Big) + \CO (\nu )  \\
=& \min \Big( 2 \alpha - \frac{8 \alpha^{2} \mu}{d} + \CO (\mu^{2}) , \frac{d}{2} \Big) + \CO (\nu ) .
\end{align*}
On the other hand,
\begin{equation*}
\frac{d}{4p} = \frac{d}{4} \min \Big( \Big( \frac{d}{4 \alpha} + \mu \Big)^{-1} ,1 \Big) = \min \Big( \alpha + \CO ( \mu ) , \frac{d}{4} \Big) .
\end{equation*}
Taking first $\mu$ and then $\nu$ small enough, the lemma follows from the estimates \eqref{b4} and \eqref{b5}.
\end{proof}

\begin{lemma}\sl \label{b47}
Let $\beta \geq 0$, $0 \leq \gamma \leq \min ( 1 , d/4)$ and $0 \leq \delta \leq d/4$. Then, for all $\varepsilon > 0$,
\begin{equation*}
\big\Vert \< x \>^{\beta} ( \lambda P_{0} +1)^{- 1} u \big\Vert \lesssim \lambda^{\beta /2  - \delta + \varepsilon} \big\Vert \< x \>^{2 \delta} u \big\Vert + \lambda^{- \gamma + \varepsilon} \big\Vert \< x \>^{\beta + 2 \gamma} u \big\Vert ,
\end{equation*}
uniformly for $\lambda \geq 1$.
\end{lemma}

\begin{remark}\sl \label{b12}
In the previous lemma, assume $\gamma + \beta /2 \leq d/4$. Then, we can chose $\delta = \gamma + \beta /2$ and we have
\begin{equation*}
\big\Vert \< x \>^{\beta} ( \lambda P_{0} +1)^{- 1} u \big\Vert \lesssim \lambda^{- \gamma + \varepsilon} \big\Vert \< x \>^{\beta + 2 \gamma} u \big\Vert ,
\end{equation*}
uniformly for $\lambda \geq 1$.
\end{remark}

\begin{proof}
Assume first that $\beta \in \N$. Using
\begin{equation*}
\Vert \< x \> u \Vert^{2} = \big\< (x^{2} +1) u , u \big\> = \sum_{j=1}^{d} \Vert x_{j} u \Vert^{2} + \Vert u \Vert^{2} ,
\end{equation*}
it is enough to estimate $\Vert x^{a} ( \lambda P_{0} +1)^{-1} u \Vert$ where $a \in \N^{d}$ is a multi-index of length less or equal to $\beta$. Since
\begin{equation*}
x_{j} ( \lambda P_{0} +1)^{-1} = ( \lambda P_{0} +1)^{-1} x_{j} - 2 \lambda^{1/2} \big( \lambda^{1/2} \partial_{j} \big) ( \lambda P_{0} +1)^{-2} ,
\end{equation*}
the operator $x^{a} ( \lambda P_{0} +1)^{-1}$ can be written as a finite sum of terms of the form
\begin{equation*}
T = \lambda^{\frac{\vert a \vert - \vert b \vert}{2}} \big( \lambda^{1/2} \partial \big)^{c} ( \lambda P_{0} +1)^{-1 - \frac{\vert a + c - b \vert}{2}} x^{b} ,
\end{equation*}
where $b , c$ are non-negative multi-indexes such that $b + c \leq a$ and $\vert a + c - b \vert = \vert a \vert + \vert c \vert - \vert b \vert$ is even. Such a term can be written as
\begin{align*}
T =& \lambda^{\frac{\vert a \vert - \vert b \vert}{2}} \big( \lambda^{1/2} \partial \big)^{c} ( \lambda P_{0} +1)^{-1 - \frac{\vert a + c - b \vert}{2}} ( \lambda P_{0} +1)^{1 + \frac{\vert a \vert - \vert b \vert}{2}} ( \lambda P_{0} +1)^{- 1 - \frac{\vert a \vert - \vert b \vert}{2}} x^{b}    \\
=& \lambda^{\frac{\vert a \vert - \vert b \vert}{2}} B ( \lambda P_{0} +1)^{- 1 - \frac{\vert a \vert - \vert b \vert}{2}} x^{b} ,
\end{align*}
where $B$ is a bounded operator on $L^{2} (\R^{d})$ since it is a Fourier multiplier by a uniformly bounded function.

Using Lemma \ref{b6} to estimate the powers of the resolvent, we get
\begin{equation} \label{b7}
T = B \lambda^{\frac{\vert a \vert - \vert b \vert}{2} - \min ( \alpha , d/4 ) + \varepsilon}  \< x \>^{\vert b \vert + \min (2 \alpha , d/2 + \varepsilon )}
\end{equation}
where $B$ is an other bounded operator, $0< \varepsilon$ and $0 \leq \alpha \leq 1 + ( \vert a \vert - \vert b \vert ) /2$. We choose $\alpha = \min ( \gamma + ( \vert a \vert - \vert b \vert ) /2 , \delta ) \leq d/4$ and note $b_{0} = \vert a \vert + 2 \gamma - 2 \delta$.

If $\vert b \vert < b_{0}$, then $\alpha = \delta$ and \eqref{b7} becomes
\begin{align*}
T =& B \lambda^{\frac{\vert a \vert - \vert b \vert}{2} - \delta + \varepsilon} \< x \>^{\vert b \vert + 2 \delta} \\
=& \CO \big( \lambda^{\frac{\vert a \vert}{2} - \delta + \varepsilon} \< x \>^{2 \delta} + \lambda^{\frac{\vert a \vert - b_{0}}{2} - \delta + \varepsilon} \< x \>^{b_{0} + 2 \delta} \big) ,
\end{align*}
since $y^{\vert b \vert} \leq y^{b_{0}} + y^{0}$ for $0 \leq b \leq b_{0}$ and $y \geq 0$. Using $\vert a \vert \leq \beta$, we get
\begin{equation} \label{b8}
T = \CO \big( \lambda^{\beta /2  - \delta + \varepsilon} \< x \>^{2 \delta} + \lambda^{- \gamma + \varepsilon} \< x \>^{\beta + 2 \gamma} \big) .
\end{equation}

If $\vert b \vert \geq b_{0}$, then $\alpha = \gamma + ( \vert a \vert - \vert b \vert ) /2$ and \eqref{b7} gives
\begin{equation} \label{b9}
T = \CO \big( \lambda^{- \gamma + \varepsilon}  \< x \>^{\beta + 2 \gamma} \big) .
\end{equation}
The estimates \eqref{b8} and \eqref{b9} imply the lemma for $\beta \in \N$. The case $\beta \in \R^{+}$ follows from an interpolation argument.
\end{proof}

Mimicking the previous proofs, one can show the following results

\begin{lemma}\sl \label{b13}
Let $j \in \{1 , \ldots , d \}$, $\beta \geq 0$ and $0 \leq \gamma \leq 1 /2$ with $\gamma + \beta /2 \leq d/4$. Then, for all $\varepsilon > 0$, we have
\begin{equation*}
\big\Vert \< x \>^{\beta} (\lambda^{1/2} \partial_{j}) ( \lambda P_{0} +1)^{- 1} u \big\Vert \lesssim \lambda^{- \gamma + \varepsilon} \big\Vert \< x \>^{\beta + 2 \gamma} u \big\Vert ,
\end{equation*}
uniformly for $\lambda \geq 1$.
\end{lemma}

\begin{lemma}\sl \label{b14}
Let $j,k \in \{1 , \ldots , d \}$ and $0 \leq \beta /2 \leq d/4$. Then, for all $\varepsilon > 0$, we have
\begin{equation*}
\big\Vert \< x \>^{\beta} (\lambda^{1/2} \partial_{j}) ( \lambda P_{0} +1)^{- 1} (\lambda^{1/2} \partial_{k}) u \big\Vert \lesssim \lambda^{\varepsilon} \big\Vert \< x \>^{\beta} u \big\Vert ,
\end{equation*}
uniformly for $\lambda \geq 1$.
\end{lemma}

\Subsection{Estimates for an intermediate operator}

We now extend these results to the intermediate differential operator $\widetilde{P}$ defined by
\begin{equation} \label{b33}
\widetilde{P} = - \sum_{j,k} \partial_{j} g^{2} g^{j,k} \partial_{k} .
\end{equation}
Recall from \eqref{c1} that $g^{2} g^{j,k} - \delta_{j,k} = \CO ( \< x \>^{- \rho })$. The square roots of $\widetilde{P}$ and $P_{0}$ are comparable. More precisely, we have

\begin{lemma}\sl \label{b16}
For $u \in D (\widetilde{P}^{1/2}) = D (P_{0}^{1/2}) = H^{1} (\R^{d})$,
\begin{equation*}
\Vert \widetilde{P}^{1/2} u \Vert \lesssim \Vert P_{0}^{1/2} u \Vert \lesssim \Vert \widetilde{P}^{1/2} u \Vert .
\end{equation*}
\end{lemma}

\begin{proof}
For $u \in H^{2} (\R^{d})$, we can write
\begin{equation*}
( \widetilde{P} u ,u) = \sum_{j,k} ( g^{2} g^{j,k} \partial_{j} u , \partial_{k} u) \quad \text{ and } \quad
(P_{0} u ,u) = \sum_{j} ( \partial_{j} u , \partial_{j} u) .
\end{equation*}
Using the ellipticity of $\widetilde{P}$ and $g^{2} g^{j,k} \in L^{\infty} (\R^{d})$, we get
\begin{equation*}
( P_{0} u ,u ) \lesssim ( \widetilde{P} u ,u) \lesssim ( P_{0} u ,u ).
\end{equation*}
In particular, we have, for $u \in H^{2} (\R^{d})$,
\begin{gather*}
\Vert \widetilde{P}^{1/2} u \Vert \lesssim \Vert P_{0}^{1/2} u \Vert \lesssim \Vert \widetilde{P}^{1/2} u \Vert  \\
\Vert ( \widetilde{P} +1)^{1/2} u \Vert \lesssim \Vert (P_{0} +1)^{1/2} u \Vert \lesssim \Vert (\widetilde{P}+1)^{1/2} u \Vert .
\end{gather*}
Then, we obtain $D (\widetilde{P}^{1/2}) = D (P_{0}^{1/2}) = H^{1} (\R^{d})$ and the lemma follows.
\end{proof}

\begin{lemma}\sl \label{b19}
Let $\beta \geq 0$ and $0 \leq \gamma \leq \min ( 1 , d/4)$ with $\gamma + \beta /2 \leq d/4$. Then, for all $\varepsilon > 0$, we have
\begin{equation} \label{b11}
\big\Vert \< x \>^{\beta} ( \lambda \widetilde{P} +1)^{- 1} u \big\Vert \lesssim \lambda^{- \gamma + \varepsilon} \big\Vert \< x \>^{\beta + 2 \gamma} u \big\Vert ,
\end{equation}
uniformly for $\lambda \geq 1$.
\end{lemma}

\begin{remark}\sl \label{b51}
Mimicking the proof of Lemma \ref{b19}, one can show that Lemma \ref{b13} (for the operators $(\lambda^{1/2} \partial_{j}) ( \lambda \widetilde{P} +1)^{- 1}$ and $( \lambda \widetilde{P} +1)^{- 1} (\lambda^{1/2} \partial_{j})$) and Lemma \ref{b14} hold with $P_{0}$ replaced by $\widetilde{P}$.
\end{remark}

\begin{proof}
From \eqref{b33}, we have
\begin{equation*}
P_{0} - \widetilde{P} = \sum_{j ,k} \partial_{j} r_{j ,k} \partial_{k} ,
\end{equation*}
where $r_{j,k} = \delta_{j ,k} - g^{2} g^{j,k} = \CO (\< x\>^{- \rho})$. In the following, to clarify the statement, we will not write the sum over $j,k$ and simply note $P_{0} - \widetilde{P} = \partial r \partial$. Iterating the resolvent identity, we have
\begin{align}
(\lambda \widetilde{P} +1)^{-1} =& (\lambda P_{0} +1 )^{-1} + (\lambda P_{0} +1 )^{-1} \lambda^{1/2} \partial r \lambda^{1/2} \partial (\lambda P_{0} +1 )^{-1}  \nonumber  \\
&+ \sum_{j=1}^{2N} (\lambda P_{0} +1 )^{-1} ( \lambda^{1/2} \partial ) \Big( r ( \lambda^{1/2} \partial ) (\lambda P_{0} +1 )^{-1} ( \lambda^{1/2} \partial ) \Big)^{j} r ( \lambda^{1/2} \partial ) (\lambda P_{0} +1 )^{-1}  \nonumber  \\
&+ (\lambda P_{0} +1 )^{-1} ( \lambda^{1/2} \partial ) \Big( r ( \lambda^{1/2} \partial ) (\lambda P_{0} +1 )^{-1} ( \lambda^{1/2} \partial ) \Big)^{N}  \nonumber \\
&\quad \times r ( \lambda^{1/2} \partial ) (\lambda \widetilde{P} +1 )^{-1} ( \lambda^{1/2} \partial ) r   \nonumber \\
&\quad \times \Big( ( \lambda^{1/2} \partial ) (\lambda P_{0} +1 )^{-1} ( \lambda^{1/2} \partial ) r \Big)^{N} ( \lambda^{1/2} \partial ) (\lambda P_{0} +1 )^{-1} .  \label{b15}
\end{align}

Thanks to Remark \ref{b12}, the first term of the previous equation satisfies \eqref{b11}. To treat the second term, we use two times Lemma \ref{b13} with a gain equal to $\gamma /2 \leq \max (1/2 , d/4)$.

The sum over $j$ can be studied in a similar way: using Lemma \ref{b13}, each exterior term $(\lambda^{1/2} \partial_{j}) ( \lambda P_{0} +1)^{- 1}$ gives a factor $\lambda^{- \gamma /2 + \widetilde{\varepsilon}}$, and, using Lemma \ref{b14}, each interior factor $( \lambda^{1/2} \partial )$ $(\lambda P_{0} +1 )^{-1} ( \lambda^{1/2} \partial )$ gives a factor $\lambda^{\widetilde{\varepsilon}}$. Then, each term in the sum over $j$ can be estimated by $\lambda^{- \gamma + (j +2) \widetilde{\varepsilon}}$. Taking $\widetilde{\varepsilon} = \varepsilon / (2 N + 2)$, each term of the sum over $j$ satisfies \eqref{b11}.

It remains to study the last term in \eqref{b15}. As usual, the first term can be estimated by Lemma \ref{b13}:
\begin{equation*}
\big\Vert \< x \>^{\beta} (\lambda P_{0} +1 )^{-1} ( \lambda^{1/2} \partial ) u \big\Vert \lesssim \lambda^{- \gamma /2 + \widetilde{\varepsilon}} \big\Vert \< x \>^{\beta + \gamma} u \big\Vert .
\end{equation*}
Now, using $r = \CO ( \< x \>^{- \rho} )$ together with Lemma \ref{b14}, we get
\begin{equation}
\big\Vert \< x \>^{\mu} r ( \lambda^{1/2} \partial ) (\lambda P_{0} +1 )^{-1} ( \lambda^{1/2} \partial ) u \big\Vert \lesssim \lambda^{\widetilde{\varepsilon}} \big\Vert \< x \>^{\max (\mu - \rho ,0)} u \big\Vert ,  \label{c2}
\end{equation}
for $\mu /2 \leq d/4 + \rho /2$. Using $N$ times the last inequality, we obtain
\begin{align}
\Big\Vert \< x \>^{\beta} (\lambda P_{0} +1 )^{-1} ( & \lambda^{1/2} \partial ) \Big( r ( \lambda^{1/2} \partial ) (\lambda P_{0} +1 )^{-1} ( \lambda^{1/2} \partial ) \Big)^{N} u \Big\Vert   \nonumber   \\
&\lesssim \lambda^{- \gamma /2 + (N +1) \widetilde{\varepsilon}} \big\Vert \< x \>^{\max (\beta + \gamma - \rho N ,0)} u \big\Vert \leq \lambda^{- \gamma /2 + (N +1) \widetilde{\varepsilon}} \Vert u \Vert , \label{c5}
\end{align}
for $N$ large enough. Using two times Lemma \ref{b16} and the functional calculus,
\begin{equation} \label{c4}
\big\Vert ( \lambda^{1/2} \partial ) (\lambda \widetilde{P} +1 )^{-1} ( \lambda^{1/2} \partial ) u \big\Vert \lesssim \Vert u \Vert .
\end{equation}
Finally, applying $N$ times \eqref{c2}, with $N$ large enough, we get
\begin{equation} \label{c3}
\Big\Vert \< x \>^{\gamma} \Big( r ( \lambda^{1/2} \partial ) (\lambda P_{0} +1 )^{-1} ( \lambda^{1/2} \partial ) \Big)^{N} \Big\Vert \lesssim \lambda^{N \widetilde{\varepsilon}} \Vert u \Vert ,
\end{equation}
since $\gamma \leq d/4$. Moreover, using $\gamma /2 \leq 1/2$ and taking the adjoint in Lemma \ref{b13}, we have
\begin{equation*}
\big\Vert \< x \>^{- \gamma} (\lambda^{1/2} \partial_{j}) ( \lambda P_{0} +1)^{- 1} u \big\Vert \lesssim \lambda^{- \gamma /2 + \widetilde{\varepsilon}} \big\Vert u \big\Vert .
\end{equation*}
Combining the last estimate with the adjoint of \eqref{c3}, it follows
\begin{equation}
\Big\Vert \Big( ( \lambda^{1/2} \partial ) (\lambda P_{0} +1 )^{-1} ( \lambda^{1/2} \partial ) r \Big)^{N} ( \lambda^{1/2} \partial ) (\lambda P_{0} +1 )^{-1} \Big\Vert \lesssim \lambda^{- \gamma /2 + (N +1) \widetilde{\varepsilon}} \Vert u \Vert , \label{c6}
\end{equation} 
for $N$ large enough. Summing up \eqref{c5}, \eqref{c4}, \eqref{c6} and choosing $\widetilde{\varepsilon}$ small enough with respect to $\varepsilon$, the last term in \eqref{b15} satisfies \eqref{b11}.
\end{proof}

\Subsection{Estimates for the perturbed Laplacian}

Here, we extend the previous results to the Laplacian $P$. From \eqref{b33}, we have $P = g^{-1} \widetilde{P} g^{-1}$. In particular, the resolvent identity gives
\begin{align}
( \lambda P +1)^{-1 } =& g ( \lambda \widetilde{P} + g^{2} )^{-1 } g    \nonumber  \\
=& g ( \lambda \widetilde{P} + 1 )^{-1 } g + g ( \lambda \widetilde{P} + 1 )^{-1 } (1-g^{2}) g^{-1} ( \lambda P + 1 )^{-1 } \label{b35}  \\
=& g ( \lambda \widetilde{P} + 1 )^{-1 } g + ( \lambda P + 1 )^{-1 } g^{-1} (1-g^{2}) ( \lambda \widetilde{P} + 1 )^{-1 } g . \label{b36}
\end{align}
Note that, by \eqref{c1}, $(1-g^{2}) = \CO ( \< x \>^{- \rho} )$.

\begin{proposition}\sl \label{b45}
Let $\beta \geq 0$ and $0 \leq \gamma \leq 1$ with $\gamma + \beta /2 \leq d/4$. Then, for all $\varepsilon > 0$, we have
\begin{equation*}
\big\Vert \< x \>^{\beta} ( \lambda P +1)^{- 1} u \big\Vert \lesssim \lambda^{- \gamma + \varepsilon} \big\Vert \< x \>^{\beta + 2 \gamma} u \big\Vert ,
\end{equation*}
uniformly for $\lambda \geq 1$.
\end{proposition}

\begin{proof}
As in the proof of Lemma \ref{b19}, we iterate the resolvent identity \eqref{b35} and obtain
\begin{align}
(\lambda P +1)^{-1} =& g ( \lambda \widetilde{P} + 1 )^{-1 } g + g ( \lambda \widetilde{P} + 1 )^{-1 } (1 -g^{2}) ( \lambda \widetilde{P} + 1 )^{-1 } g  \nonumber  \\
&+ \sum_{j=1}^{N} g ( \lambda \widetilde{P} + 1 )^{-1 } \Big( (1-g^{2}) ( \lambda \widetilde{P} + 1 )^{-1 } \Big)^{j} (1 -g^{2}) ( \lambda \widetilde{P} + 1 )^{-1 } g  \nonumber  \\
&+ g ( \lambda \widetilde{P} + 1 )^{-1 } \Big( (1-g^{2}) ( \lambda \widetilde{P} + 1 )^{-1 } \Big)^{N+1} (1-g^{2}) g^{-1} ( \lambda P +1)^{-1} . \label{b37}
\end{align}
The two first terms and the sum over $j$ can be directly estimated by Lemma \ref{b19}. For the last term in \eqref{b37}, we remark that Lemma \ref{b19} gives
\begin{equation} \label{b38}
\big\Vert \< x \>^{\beta} g ( \lambda \widetilde{P} + 1 )^{-1 } u \big\Vert \lesssim \lambda^{- \gamma + \widetilde{\varepsilon}} \big\Vert \< x \>^{\beta + 2 \gamma} u \big\Vert ,
\end{equation}
and
\begin{equation} \label{b39}
\big\Vert \< x \>^{\mu} (1-g^{2}) ( \lambda \widetilde{P} + 1 )^{-1 } u \big\Vert \lesssim \lambda^{\widetilde{\varepsilon}} \big\Vert \< x \>^{\max ( \mu - \rho ,0)} u \big\Vert ,
\end{equation}
for all $\mu /2 \leq d/4 + \rho /2$. Therefore, applying \eqref{b38} and $N+1$ times \eqref{b39} (this can be made since $( \beta + 2 \gamma ) /2 \leq d/4$), we get
\begin{align*}
\Big\Vert \< x \>^{\beta} g ( \lambda \widetilde{P} + 1 )^{-1 } \Big( (1-g^{2}) ( \lambda \widetilde{P} + 1 )^{-1 } \Big)^{N+1} u \Big\Vert &\lesssim \lambda^{- \gamma + (N +2 ) \widetilde{\varepsilon}} \big\Vert \< x \>^{\max ( \mu - (N+1) \rho ,0)} u \big\Vert  \\
&\lesssim \lambda^{- \gamma + (N +2 ) \widetilde{\varepsilon}} \Vert u \Vert ,
\end{align*}
for $N$ large enough. Using $\Vert ( \lambda P + 1 )^{-1 } \Vert \leq 1$ by the spectral theorem and taking $\widetilde{\varepsilon} = \varepsilon / ( N+2)$, this implies
\begin{equation*}
\Big\Vert \< x \>^{\beta} g ( \lambda \widetilde{P} + 1 )^{-1 } \Big( (1-g^{2}) ( \lambda \widetilde{P} + 1 )^{-1 } \Big)^{N+1} (1-g^{2}) g^{-1} ( \lambda P +1)^{-1} u \Big\Vert \lesssim \lambda^{- \gamma + \varepsilon} \Vert u \Vert ,
\end{equation*}
and the lemma follows.
\end{proof}

Mimicking the proof of Proposition \ref{b45} and using \eqref{b36} and Remark \ref{b51}, one can prove, as for Lemma \ref{b13}, the following result.

\begin{lemma}\sl \label{b20}
Let $j \in \{1 , \ldots , d \}$, $\beta \geq 0$ and $0 \leq \gamma \leq 1 /2$ with $\gamma + \beta /2 \leq d/4$. Then, for all $\varepsilon > 0$, we have
\begin{align*}
\big\Vert \< x \>^{\beta} ( \lambda P +1)^{- 1} (\lambda^{1/2} \widetilde{\partial}_{j}^{*} ) u \big\Vert & \lesssim \lambda^{- \gamma + \varepsilon} \big\Vert \< x \>^{\beta + 2 \gamma} u \big\Vert \\
\big\Vert \< x \>^{\beta} (\lambda^{1/2} \widetilde{\partial}_{j} ) ( \lambda P +1)^{- 1} u \big\Vert & \lesssim \lambda^{- \gamma + \varepsilon} \big\Vert \< x \>^{\beta + 2 \gamma} u \big\Vert .\end{align*}
uniformly for $\lambda \geq 1$.

Let $j,k \in \{1 , \ldots , d \}$ and $0 \leq \beta /2 \leq d/4$. Then, for all $\varepsilon > 0$, we have
\begin{equation*}
\big\Vert \< x \>^{\beta} (\lambda^{1/2} \widetilde{\partial}_{j} ) ( \lambda P +1)^{- 1} (\lambda^{1/2} \widetilde{\partial}_{k}^{*} ) u \big\Vert \lesssim \lambda^{\varepsilon} \big\Vert \< x \>^{\beta} u \big\Vert ,
\end{equation*}
uniformly for $\lambda \geq 1$.
\end{lemma}

\begin{remark}\sl \label{b59}
The results of this section are given for $( \lambda P +1 )^{-1}$, but can be extended to $( \lambda P - z )^{-1}$, with $\im z \neq 0$. In fact, following the previous proofs, one can see that $( \lambda P - z )^{-1}$ satisfies the same results, if we accept a lose of the form $\vert \im z \vert^{-C}$, $C>0$, in the estimates. This is due to $(\lambda P_{0} +1) ( \lambda P_{0} -z)^{-1} = \CO ( \vert \im z \vert^{-1} )$ from the spectral theorem. Note that the constant $C$ does not depend on $\varepsilon \in ]0,1]$, and is uniform with respect to $\alpha , \beta, \gamma , \delta$ in a compact subset.

For example, Proposition \ref{b45} gives the following estimate for $\beta \geq 0$, $\varepsilon >0$ and $0 \leq \gamma \leq 1$ with $\gamma + \beta /2 \leq d/4$:
\begin{equation} \label{b60}
\big\Vert \< x \>^{\beta} ( \lambda P - z )^{- 1} u \big\Vert \lesssim \frac{\lambda^{- \gamma + \varepsilon}}{\vert \im z \vert^{C}}  \big\Vert \< x \>^{\beta + 2 \gamma} u \big\Vert ,
\end{equation}
uniformly for $\lambda \geq 1$ and $z$ in a compact of $\C$. 
\end{remark}

Using the spectral theorem, this remark implies the following result.

\begin{lemma}\sl \label{b61}
Let $\chi \in C^{\infty}_{0} ( \R )$, $j,k \in \{1 , \ldots , d \}$ and $\beta , \gamma \geq 0$ with $\gamma + \beta /2 \leq d/4$. Then, for all $\varepsilon > 0$, we have
\begin{gather*}
\big\Vert \< x \>^{\beta} \chi ( \lambda P ) u \big\Vert \lesssim \lambda^{- \gamma + \varepsilon} \big\Vert \< x \>^{\beta + 2 \gamma} u \big\Vert   \\
\big\Vert \< x \>^{\beta} (\lambda^{1/2} \widetilde{\partial}_{j}) \chi ( \lambda P ) u \big\Vert \lesssim \lambda^{- \gamma + \varepsilon} \big\Vert \< x \>^{\beta + 2 \gamma} u \big\Vert    \\
\big\Vert \< x \>^{\beta} \chi ( \lambda P ) (\lambda^{1/2} \widetilde{\partial}_{j}^{*} ) u \big\Vert \lesssim \lambda^{- \gamma + \varepsilon} \big\Vert \< x \>^{\beta + 2 \gamma} u \big\Vert    \\
\big\Vert \< x \>^{\beta} (\lambda^{1/2} \widetilde{\partial}_{j}) \chi ( \lambda P ) ( \lambda^{1/2} \widetilde{\partial}_{k}^{*} ) u \big\Vert \lesssim \lambda^{- \gamma + \varepsilon} \big\Vert \< x \>^{\beta + 2 \gamma} u \big\Vert ,
\end{gather*}
uniformly for $\lambda \geq 1$.
\end{lemma}

\begin{proof}
We only prove the first inequality since the others can be treated the same way. Let $k \in \N$ be such that $\gamma / k \leq 1$, $\varphi ( \sigma ) = \chi ( \sigma ) ( \sigma +1 )^{k} \in C^{\infty}_{0} (\R )$ and $\widetilde{\varphi} \in C_{0}^{\infty} ( \C )$ be an almost analytic extension of $\varphi$. From the spectral theorem, we have
\begin{equation} \label{b63}
\< x \>^{\beta} \chi ( \lambda P )= \frac{1}{\pi} \int \overline{\partial} \widetilde{\varphi} (z) \< x \>^{\beta} ( \lambda P -z)^{-1} ( \lambda P +1)^{-k} L (d z) .
\end{equation}
Estimate \eqref{b60} with $\gamma = 0$ gives
\begin{equation} \label{b62}
\big\Vert \< x \>^{\beta} (\lambda P -z)^{-1} u \big\Vert \lesssim \frac{\lambda^{\widetilde{\varepsilon}}}{\vert \im z \vert^{C}} \big\Vert \< x \>^{\beta} u \big\Vert .
\end{equation}
Proposition \ref{b45} with $\gamma = \gamma / k \leq 1$ implies
\begin{equation*}
\big\Vert \< x \>^{\mu} ( \lambda P +1)^{-1} u \big\Vert \lesssim \lambda^{- \gamma / k + \widetilde{\varepsilon}} \big\Vert \< x \>^{\mu + 2 \gamma / k} u \big\Vert ,
\end{equation*}
if $\gamma / k + \mu /2 \leq d/4$. By iteration, we obtain
\begin{equation*}
\big\Vert \< x \>^{\beta} ( \lambda P +1)^{-k} u \big\Vert \lesssim \lambda^{- \gamma + k \widetilde{\varepsilon}} \big\Vert \< x \>^{\beta + 2 \gamma} u \big\Vert ,
\end{equation*}
since $\gamma  + \beta /2 \leq d/4$. Combining this estimate with \eqref{b62} and taking $\widetilde{\varepsilon} = \varepsilon / (k+1)$, we get
\begin{equation*}
\big\Vert \< x \>^{\beta} (\lambda P -z)^{-1} (\lambda P +1)^{-k} u \big\Vert \lesssim \frac{\lambda^{- \gamma + \varepsilon}}{\vert \im z \vert^{C}} \big\Vert \< x \>^{\beta + 2 \gamma} u \big\Vert ,
\end{equation*}
and the lemma follows from \eqref{b63}.
\end{proof}

We now state a result which will help us to estimate the square root of $P$. Since this lemma can be proved as Lemma \ref{b16}, we do not give the proof.

\begin{lemma}\sl \label{b53}
We have, for $u \in D (P^{1/2}) = H^{1} (\R^{d})$,
\begin{equation*}
\Vert P^{1/2} u \Vert \lesssim \Vert \nabla g^{-1} u \Vert \lesssim \Vert P^{1/2}  u \Vert .
\end{equation*}
\end{lemma}

\bibliographystyle{amsplain}
\providecommand{\bysame}{\leavevmode\hbox to3em{\hrulefill}\thinspace}
\providecommand{\MR}{\relax\ifhmode\unskip\space\fi MR }
\providecommand{\MRhref}[2]{%
  \href{http://www.ams.org/mathscinet-getitem?mr=#1}{#2}
}
\providecommand{\href}[2]{#2}

%\bibliographystyle{amsplain}
%\bibliography{refdd}

\end{document}